\documentclass[twoside,a4paper,reqno,11pt]{amsart}
\usepackage{amsfonts, amsbsy, amsmath, amsthm, amssymb, latexsym, verbatim, enumerate}
\usepackage{mathrsfs}

\usepackage[top=30mm,right=30mm,bottom=30mm,left=30mm]{geometry}

\usepackage{bm}
\usepackage[pdftex]{color,graphicx}

\makeatletter
\newcommand{\imod}[1]{\allowbreak\mkern4mu({\operator@font mod}\,\,#1)}
\makeatother

\renewcommand{\a}{\alpha}

 \newcommand{\e}{\epsilon}
\renewcommand{\bf}{\textbf} 
\renewcommand{\l}{\lambda} 
 
 \renewcommand{\to}{\rightarrow}

\newcommand{\leqs}{\leqslant}
\newcommand{\geqs}{\geqslant}

 \newcommand{\vs}{\vspace{3mm}}
\newcommand{\la}{\langle}
\newcommand{\ra}{\rangle}

\newcommand{\Z}{\mathbb{Z}}

\newtheorem{theorem}{Theorem}

\newtheorem{corol}{Corollary}

\newtheorem{thm}{Theorem}[section]
\newtheorem{prop}[thm]{Proposition}
\newtheorem{lem}[thm]{Lemma}
\newtheorem{cor}[thm]{Corollary}

\theoremstyle{definition}
\newtheorem{ex}[thm]{Example}

\newtheorem{remk}[thm]{Remark}

\newtheorem{remark}{Remark}

\begin{document}

\author{Timothy C. Burness}
 \address{T.C. Burness, School of Mathematics, University of Bristol, Bristol BS8 1TW, UK}
 \email{t.burness@bristol.ac.uk}

\author{Donna M. Testerman}
\address{D.M. Testerman, Institute of Mathematics, Station 8, \'{E}cole Polytechnique F\'{e}d\'{e}rale de Lausanne, CH-1015 Lausanne, Switzerland}
\email{donna.testerman@epfl.ch}

\title[$A_1$-type subgroups containing regular unipotent elements]{$A_1$-type subgroups containing regular \\ unipotent elements}

\subjclass[2010]{Primary 20G41; Secondary 20E32, 20E07}

\begin{abstract}
Let $G$ be a simple exceptional algebraic group of adjoint type over an algebraically closed field of characteristic $p>0$ and let $X = {\rm PSL}_{2}(p)$ be a subgroup of $G$ containing a regular unipotent element $x$ of $G$. By a theorem of Testerman, $x$ is contained in a connected subgroup of $G$ of type $A_1$. In this paper we prove that with two exceptions, $X$ itself is contained in such a subgroup (the exceptions arise when $(G,p) = (E_6,13)$ or $(E_7,19)$). This extends earlier work of Seitz and Testerman, who established the containment under some additional conditions on $p$ and the embedding of $X$ in $G$. We discuss applications of our main result to the study of the subgroup structure of finite groups of Lie type.  
\end{abstract}

\date{\today}
\maketitle

\section{Introduction}\label{s:intro}

Let $G$ be a simple algebraic group of adjoint type over an algebraically closed field $K$ of characteristic $p>0$. Let $X = {\rm PSL}_{2}(q)$ be a subgroup of $G$, where $q \geqs 4$ is a $p$-power, and let $x \in X$ be an element of order $p$. By the main theorem of \cite{PST}, $x$ is contained in a closed connected subgroup of $G$ of type $A_1$, unless $G=G_2$, $p=3$ and $x$ belongs to the conjugacy class of $G$ labelled $A_1^{(3)}$ as in \cite{Lawther}. With a view towards applications to the study of the subgroup structure of finite groups of Lie type, it is desirable to seek natural extensions of this result. In particular, under what conditions can one embed the full subgroup $X$ in an $A_1$ subgroup of $G$?

As a special case of the main theorem of \cite{ST1}, this question has a positive answer when $G$ is classical and $X$ is not contained in a proper parabolic subgroup of $G$ (for $G={\rm SL}_{n}(K)$, this is a well-known theorem of Steinberg \cite{St}). One can see that the condition on the embedding of $X$ is necessary 
by considering indecomposable representations of $X$ which do not
arise as restrictions of indecomposable representations of an algebraic $A_1$. In \cite{ST1}, Seitz and Testerman also provide a positive answer if $G$ is a simple exceptional algebraic group (of type $G_2$, $F_4$, $E_6$, $E_7$ or $E_8$) and $p$ is large enough, still under the same assumption that $X$ is not contained in a proper parabolic subgroup of $G$. More precisely, the approach in \cite{ST1} requires $p > N(G)$ where 
\begin{equation}\label{e:Ndef}
N(G_2) = 19,\; N(F_4) = 43,\; N(E_6) = 43,\; N(E_7) = 67, \; N(E_8) = 113.
\end{equation}
More general results on the embedding of finite quasisimple subgroups in exceptional algebraic groups are established by Liebeck and Seitz in \cite{LS98}. For instance, if $X = {\rm PSL}_{2}(q)$ and $q$ is sufficiently large, then \cite[Theorem 1]{LS98} implies that $X$ is contained in a proper closed positive dimensional subgroup of $G$. Here ``sufficiently large" means that $q>t(G)\cdot (2,p-1)$ with  
\begin{equation}\label{e:tdef}
t(G_2)=12, \; t(F_4) = 68, \; t(E_6)=124, \; t(E_7) = 388,\; t(E_8) = 1312.
\end{equation}

It is natural to seek an extension of \cite[Theorem 2]{ST1} by removing the conditions on $p$ and the embedding of $X$ in $G$ when $G$ is of exceptional type and $X = {\rm PSL}_{2}(q)$. In \cite{ST2}, Seitz and Testerman study the case where $x \in X$ is \emph{semiregular} in $G$ (that is, $C_G(x)$ is a unipotent group). Notice that if $x$ is \emph{not} semiregular then $x \in C_G(s)^{0}$ for some non-trivial semisimple element $s \in G$ and one can hope to answer the question in the proper reductive subgroup $C_G(s)^{0}$; so the semiregular case, where such a reduction is not possible, is particularly interesting. In this situation, the main result of \cite{ST2} states that $X$ is contained in a connected subgroup of type $A_1$ if either $q>p$, or if $q=p$ and ${\rm PGL}_{2}(q) \leqs N_G(X)$.  

In this paper, we extend the results in \cite{ST2} by studying the remaining case where $X = {\rm PSL}_{2}(p)$ and ${\rm PGL}_{2}(p) \not\leqs N_G(X)$. In order to do this, we will assume $x \in X$ is \emph{regular} in $G$, which means that $C_G(x)$ is an abelian unipotent group of dimension $r$, where $r$ is the rank of $G$ (equivalently, $x$ is contained in a unique Borel subgroup of $G$). It is well known that regular unipotent elements exist in all characteristics and they form a single conjugacy class. Since the order of $x$ is the smallest power of $p$ greater than the height of the highest root of $G$ (see \cite[Order Formula 0.4]{T}), our hypothesis implies that $p \geqs h$, where $h$ is the Coxeter number of $G$. (Recall that $h = \frac{1}{r}\dim G - 1 = {\rm ht}(\a_0)+1$, where ${\rm ht}(\a_0)$ is the height of the highest root of $G$.) 

Our main result is the following (in this paper, an \emph{$A_1$-type subgroup} is a closed connected subgroup isomorphic to ${\rm SL}_{2}(K)$ or ${\rm PSL}_{2}(K)$).

\begin{theorem}\label{t:main}
Let $G$ be a simple exceptional algebraic group of adjoint type over an algebraically closed field of characteristic $p >0$. Let $X = {\rm PSL}_{2}(p)$ be a subgroup of $G$ containing a regular unipotent element of $G$. Then exactly one of the following holds:
\begin{itemize}\addtolength{\itemsep}{0.2\baselineskip}
\item[{\rm (i)}] $X$ is contained in an $A_1$-type subgroup of $G$;
\item[{\rm (ii)}] $G=E_6$, $p=13$ and $X$ is contained in a $D_5$-parabolic subgroup of $G$;
\item[{\rm (iii)}] $G=E_7$, $p=19$ and $X$ is contained in an $E_6$-parabolic subgroup of $G$. 
\end{itemize}
In all three cases, $X$ is uniquely determined up to $G$-conjugacy.
\end{theorem}

\begin{remark}\label{r:main}
Let us make some comments on the statement of Theorem \ref{t:main}. 
\begin{itemize}\addtolength{\itemsep}{0.2\baselineskip}
\item[{\rm (a)}] To see the uniqueness of $X$ in part (i), it suffices to show that  every subgroup $Y = {\rm PSL}_{2}(p)$ of $G$ containing $x$ is conjugate to $X$. Write $X<A$ and $Y<B$, where $A$ and $B$ are $A_1$-type subgroups of $G$. By Proposition \ref{p:0}(ii), $A$ and $B$ are $G$-conjugate, say $A = B^g$, so $X,Y^g < A$. Finally, by applying \cite[Theorem 5.1]{LS94} and Lang's theorem, we deduce that $X$ and $Y^g$ are $A$-conjugate.
\item[{\rm (b)}] The interesting examples arising in (ii) and (iii) were found by Craven \cite{Craven} in his recent study of the maximal subgroups with socle ${\rm PSL}_{2}(q)$ in finite exceptional groups of Lie type. The action of such a subgroup $X$ on the adjoint module ${\rm Lie}(G)$ is described in Theorem \ref{t:par} (see Section \ref{s:final}) and its construction is explained in  \cite[Section 9]{Craven}. Let us say a few words on the construction in (ii), where $G = E_6$ and $p=13$. Let $P=QL$ be a $D_5$-parabolic subgroup of $G$ and identify the unipotent radical $Q$ with a $16$-dimensional spin module for $L'=D_5$. Take a subgroup $Y = {\rm PSL}_{2}(p)<L'$ containing a regular unipotent element of $L'$ 
and consider the semidirect product $QY<P$ (note that $Y$ is uniquely determined up to $L'$-conjugacy). Now one checks that $Q|_{Y}$ has an $11$-dimensional composition factor $W$ with $\dim H^1(Y,W) =1$, which is a direct summand of $Q$. It follows that there is a complement $X = {\rm PSL}_{2}(p)$ to $Q$ in $QY$ that is not $QY$-conjugate to $Y$. Moreover, one can show that $X$ contains a regular unipotent element of $G$ and there is a unique $P$-class of such subgroups $X$ (hence $X$ is uniquely determined up to $G$-conjugacy). We will show that the subgroup $X$ constructed in this way is not contained in an $A_1$-type subgroup of $G$ (this follows from  Theorem \ref{t:main0} below). A similar construction can be given in (iii) and again one can show that such a subgroup is both unique up to conjugacy and is not contained in an $A_1$-type subgroup.  
\item[{\rm (c)}] The conclusion of Theorem \ref{t:main} for $G=G_2$ can be deduced from the proof of \cite[Lemma 3.1]{ST2}. It also follows from Kleidman's classification of the maximal subgroups of $G_2(p)$ in \cite{K}. However, for completeness we will provide an alternative proof, following the same approach we use for the other exceptional groups.
\item[{\rm (d)}] Finally, let us comment on the adjoint hypothesis in the statement of the theorem. Let $G$ be a simple exceptional algebraic group and let $G_{{\rm ad}}$ be the corresponding adjoint group. Suppose $Y = {\rm PSL}_{2}(p)$ or ${\rm SL}_{2}(p)$ is a subgroup of $G$ containing a regular unipotent element $y$ of $G$. The regularity of $y$ implies that $Z(Y) \leqs Z(G)$ and thus $YZ(G)/Z(G) = {\rm PSL}_{2}(p)$ is a subgroup of $G_{{\rm ad}}$ containing a regular unipotent element, so it is determined by Theorem \ref{t:main}.  
\end{itemize}
\end{remark}

The next result shows that the subgroups $X$ in part (i) of Theorem \ref{t:main} are \emph{$G$-irreducible} in the sense of Serre (that is, $X$ is not contained in a proper parabolic subgroup of $G$). The proof is given at the end of Section \ref{s:prel}. By \cite[Theorem 1.2]{TZ}, any connected reductive subgroup of a reductive algebraic group $G$ containing a regular unipotent element is $G$-irreducible, so we can view Theorem \ref{t:main0} as a partial analogue for subgroups isomorphic to ${\rm PSL}_{2}(p)$ in simple exceptional groups.

\begin{theorem}\label{t:main0}
Let $G$ be a simple exceptional algebraic group of adjoint type over an algebraically closed field of characteristic $p >0$ and let $x \in G$ be a regular unipotent element such that  
$$x \in X = {\rm PSL}_{2}(p)<A<G,$$ 
where $A$ is an $A_1$-type subgroup of $G$. Then $X$ is $G$-irreducible.
\end{theorem}

\begin{remark}\label{r:main00}
As in Theorem \ref{t:main}, let $X = {\rm PSL}_{2}(p)$ be a subgroup of $G$ containing a regular unipotent element. By combining Theorems \ref{t:main} and \ref{t:main0}, we deduce that $X$ is contained in an $A_1$-type subgroup of $G$ if and only if $X$ is not contained in a proper parabolic subgroup of $G$. In particular, the examples arising in parts (ii) and (iii) of Theorem \ref{t:main} are genuine exceptions to the containment in (i). 
\end{remark}

The next result follows by combining Theorem \ref{t:main} with the main results of \cite{ST1,ST2}.

\begin{corol}\label{c:main}
Let $G$ be a simple algebraic group of adjoint type over an algebraically closed field of characteristic $p>0$ and let $X = {\rm PSL}_{2}(q)$ be a subgroup of $G$ containing a regular unipotent element of $G$, where $q \geqs 4$ is a $p$-power. In addition, if $G$ is classical assume that $X$ is $G$-irreducible. Then either 
\begin{itemize}\addtolength{\itemsep}{0.2\baselineskip}
\item[{\rm (a)}] $X$ is contained in an $A_1$-type subgroup of $G$, or 
\item[{\rm (b)}] $q=p$ and $(G,p,X)$ is one of the cases in parts {\rm (ii)} and {\rm (iii)} in Theorem $\ref{t:main}$.
\end{itemize}
\end{corol}

Next we present some further applications of Theorem \ref{t:main}. Let $G$ be a simple algebraic group as in Theorem \ref{t:main} and recall that a finite subgroup $H$ of $G$ is \emph{Lie primitive} if 
\begin{itemize}\addtolength{\itemsep}{0.2\baselineskip}
\item[{\rm (a)}] $H$ does not contain a subgroup of the form $O^{p'}(G^{F})$, where $F$ is a Steinberg endomorphism of $G$ with fixed point subgroup $G^F$; and
\item[{\rm (b)}] $H$ is not contained in a proper closed subgroup of $G$ of positive dimension. 
\end{itemize}

In \cite[Section 3]{GM}, Guralnick and Malle determine the maximal Lie primitive subgroups $H$ of $G$ containing a regular unipotent element (the maximal closed positive dimensional subgroups of $G$ containing a regular unipotent element were determined in earlier work of Saxl and Seitz \cite{SS}). More precisely, they give a list of possibilities for $H$, but they do not claim that all cases actually occur. In particular, their proof relies on  \cite{ST1} and thus $H = {\rm PSL}_{2}(p)$ arises as a possibility when $G \in \{F_4, E_6,E_7,E_8\}$ and $h \leqs p \leqs N(G)$, where $N(G)$ is the integer in (\ref{e:Ndef}). Therefore, by combining \cite[Theorems 3.3, 3.4]{GM} with Theorem \ref{t:main}, we obtain the following refinement.

\begin{corol}\label{c:gm}
Let $G$ be a simple exceptional algebraic group of adjoint type over an algebraically closed field of characteristic $p >0$. Suppose $H$ is a maximal Lie primitive subgroup of $G$ containing a regular unipotent element. Let $H_0$ denote the socle of $H$.
\begin{itemize}\addtolength{\itemsep}{0.2\baselineskip}
\item[{\rm (i)}] If $G = G_2$, then one of the following holds:
\begin{itemize}\addtolength{\itemsep}{0.2\baselineskip}
\item[{\rm (a)}] $p=2$ and $H = {\rm J}_2$;
\item[{\rm (b)}] $p=7$ and $H = 2^3.{\rm SL}_{3}(2)$, $G_2(2)$ or ${\rm PSL}_{2}(13)$;
\item[{\rm (c)}] $p=11$ and $H={\rm J}_{1}$.
\end{itemize}

\item[{\rm (ii)}] If $G = F_4$, then one of the following holds:
\begin{itemize}\addtolength{\itemsep}{0.2\baselineskip}
\item[{\rm (a)}] $p=2$ and $H_0 = {\rm PSL}_{3}(16)$, ${\rm PSU}_{3}(16)$ or ${\rm PSL}_{2}(17)$;
\item[{\rm (b)}] $p=13$ and $H = 3^3.{\rm SL}_{3}(3)$, or $H_0 = {\rm PSL}_{2}(25)$, ${\rm PSL}_{2}(27)$ or ${}^3D_4(2)$.
\end{itemize}

\item[{\rm (iii)}] If $G = E_6$, then one of the following holds:
\begin{itemize}\addtolength{\itemsep}{0.2\baselineskip}
\item[{\rm (a)}] $p=2$ and $H_0 = {\rm PSL}_{3}(16)$, ${\rm PSU}_{3}(16)$ or ${\rm Fi}_{22}$;
\item[{\rm (b)}] $p=13$ and either $H = 3^{3+3}.{\rm SL}_{3}(3)$ or $H_0 = {}^2F_4(2)'$.
\end{itemize}

\item[{\rm (iv)}] If $G = E_7$, then $p=19$ and $H_0 = {\rm PSU}_{3}(8)$ or ${\rm PSL}_{2}(37)$. 

\item[{\rm (v)}] If $G = E_8$, then one of the following holds:
\begin{itemize}\addtolength{\itemsep}{0.2\baselineskip}
\item[{\rm (a)}] $p=2$ and $H_0 = {\rm PSL}_{2}(31)$; 
\item[{\rm (b)}] $p=7$ and $H_0 = {\rm PSp}_{8}(7)$ or $\Omega_9(7)$; 
\item[{\rm (c)}] $p=31$ and $H = 2^{5+10}.{\rm SL}_{5}(2)$ or $5^3.{\rm SL}_{3}(5)$, or $H_0 = {\rm PSL}_{2}(32)$, ${\rm PSL}_{2}(61)$ or ${\rm PSL}_{3}(5)$.
\end{itemize}
\end{itemize}
\end{corol}

\begin{remark}
By Corollary \ref{c:gm}, there are no Lie primitive subgroups containing a regular unipotent element if $p>31$. This lower bound is best possible: the case $(G,p)= (E_8,31)$ with $H_0 = {\rm PSL}_{2}(32)$ is a genuine example (this can be deduced from recent work of Litterick \cite{AJL_mem}). However, we are not claiming that all of the possibilities listed in Corollary \ref{c:gm} are Lie primitive and contain regular unipotent elements (indeed, we expect that this list can be reduced further).
\end{remark}

We can also use Theorem \ref{t:main} to shed new light on the subgroup structure of finite exceptional groups of Lie type. Let $G$ be a simple exceptional algebraic group of adjoint type over $\bar{\mathbb{F}}_{p}$ with $p$ prime and let $F:G \to G$ be a Steinberg endomorphism of $G$ with fixed point subgroup $G^F$, an almost simple group over $\mathbb{F}_{q}$. 
The maximal subgroups of the Ree groups ${}^2G_2(q)$ and ${}^2F_4(q)$ (and their automorphism groups) have been determined up to conjugacy by Kleidman \cite{K} and Malle \cite{Mal}, respectively, and similarly $G_2(q)$ is handled in \cite{Coop} for $q$ even and in \cite{K} for $q$ odd. Therefore, we may assume $G^F$ is one of $F_4(q)$, $E_6(q)$, $^2E_6(q)$, $E_7(q)$ and $E_8(q)$.
In these cases, through the work of many authors, the problem of determining the maximal subgroups $H$ of $G^F$ has essentially been reduced to the case where $H$ is an almost simple group of Lie type with socle $H_0$ over a field $\mathbb{F}_{q_0}$ of characteristic $p$ (see \cite[Section 29.1]{MT} and the references therein). Here one of the main problems is to determine if such a subgroup is of the form $M^F$, where $M$ is maximal among positive dimensional $F$-stable closed subgroups of $G$. Significant restrictions on the rank of $H_0$ and the size of $q_0$ are established in \cite{LST, LS98}, but the problem of obtaining a complete classification is still open.   

The case $H_0={\rm PSL}_{2}(q_0)$ is of particular interest. If $q_0 > t(G) \cdot (2,p-1)$, where $t(G)$ is the integer in (\ref{e:tdef}), then the aforementioned work of Liebeck and Seitz \cite{LS98} shows that $q=q_0$ and $H = M^F$ for some maximal connected subgroup $M$ of $G$ of type $A_1$. Further results in this direction have recently been obtained by Craven \cite{Craven} when $G^F$ is one of $F_4(q)$, $E_6(q)$, ${}^2E_6(q)$ or $E_7(q)$. Using the maximality of $H$, he proves that $H = M^F$ in almost every case, but his approach is unable to eliminate certain values of $q_0$. In particular, the case where $H={\rm PSL}_{2}(h+1)$ contains a regular unipotent element of $G$ is problematic (the existence of such subgroups, in a much more general setting, was established by Serre \cite{Serre}, which explains why they are called \emph{Serre embeddings} in \cite{Craven}). Using Theorem \ref{t:main}, one can show that all maximal Serre embeddings are of the form $M^F$ (we can also handle $G=E_8$, which is excluded in \cite{Craven}). In particular, it follows that part (1) in \cite[Theorem 1.2]{Craven} is a subcase of part (2), and similarly part (2) in \cite[Theorem 1.4]{Craven} is a subcase of part (3).

\vs

To conclude the introduction, let us briefly describe the main steps in the proof of Theorem \ref{t:main} (we refer the reader to Section \ref{ss:me} for more details). Suppose $x \in X = {\rm PSL}_{2}(p)<G$ is a regular unipotent element of $G$ and let $A<G$ be an $A_1$-type subgroup containing $x$ with maximal torus $T=\{t(c)  \mid c \in K^{\times}\}$. Set $V = {\rm Lie}(G)$ and $\mathbb{F}_{p}^{\times} = \la \xi\ra$. Without loss of generality, replacing $X$ by a suitable $G$-conjugate, we show that we may assume $X$ contains the toral element $t(\xi) \in T$, which corresponds  to a diagonalizable element $s \in {\rm SL}_{2}(p)$ with eigenvalues $\xi$ and $\xi^{-1}$ (see Lemma \ref{l:key}). We can use the known action of $A$ on $V$ to determine the eigenvectors and eigenspaces of $s$ on $V$ and this severely restricts the possibilities for $V|_{X}$. It is possible to obtain further restrictions on the indecomposable summands of $V|_{X}$ by considering the trace on $V$ of semisimple elements in $X$ of small order (typically, we only need to work with elements of order $2$ and $3$). 

In this way, in almost all cases, we are able to reduce to the situation where $V|_{X}$ is compatible with the action of a ${\rm PSL}_{2}(p)$ subgroup of $A$. In this situation, $V|_{X}$ is given in Table \ref{t:dec} (our notation for indecomposable summands in Table \ref{t:dec} is explained in Section \ref{s:rep}) and we observe that the socle of $V|_{X}$ has a $3$-dimensional simple summand 
$$W = \la w_2, w_0, w_{-2} \ra,$$ 
where $w_i$ is an eigenvector for $s$ with eigenvalue $\xi^i$. Let $E_i$ be the $\xi^i$-eigenspace of $s$ on $V$. Without loss of generality, we may assume that the action of $x$ on $W$ (in terms of this basis) is given by the matrix
$$\left(\begin{array}{ccc}
1 & 1 & 1 \\
0 & 1 & 2 \\
0 & 0 & 1 
\end{array}\right)$$
and thus 
\begin{eqnarray*}
w_2  \!\!\! & \in & \!\!\! \ker(x-1) \cap E_2,\\
w_0  \!\!\! & \in & \!\!\! \left(\ker((x-1)^2) \setminus \ker(x-1)\right) \cap E_0, \\
w_{-2} \!\!\! & \in & \!\!\! \left(\ker((x-1)^3) \setminus \ker((x-1)^2)\right) \cap E_{-2}.
\end{eqnarray*} 
Our main goal is to show that $W$ is an $\mathfrak{sl}_2$-subalgebra of $V$. 

To do this, we may assume that $x$ is obtained by exponentiating the regular nilpotent element $\sum_{\gamma \in \Pi(G)} e_{\gamma} \in V$ with respect to a fixed Chevalley basis 
$$\mathcal{B} = \{e_{\a}, f_{\a}, h_{\gamma} \mid \a \in \Phi^{+}(G), \, \gamma \in \Pi(G) \}$$
for $V$ (see Section \ref{ss:me} for more details). This allows us to explicitly identify a maximal torus of an $A_1$-type subgroup of $G$ containing $x$, which means that we can compute eigenvectors and eigenspaces for $s$ in terms of the Chevalley basis. With the aid of {\sc Magma} \cite{magma} to simplify the computations, we can describe the action of $x$ on $V$ in terms of a $\dim G \times \dim G$ matrix with respect to $\mathcal{B}$ and then compute bases for the subspaces $\ker((x-1)^i)$ for $i \geqs 1$. In this way, we obtain expressions for $w_2, w_0$ and $w_{-2}$ in terms of $\mathcal{B}$, but with undetermined coefficients. We then derive relations between these coefficients by considering the action of $x$ on $W$, and further relations can be found by using the fact that $C_V(x) = \ker(x-1)$ is an abelian subalgebra. Apart from a handful of special cases, this allows us to reduce to the case where $W$ is an 
$\mathfrak{sl}_2$-subalgebra and we complete the argument by showing that the stabilizer of $W$ in $G$ is an $A_1$-type subgroup.

This process of elimination and extension comprises the bulk of the proof of Theorem \ref{t:main} (see Sections \ref{s:g2}--\ref{s:e8}). However, there are a handful of possibilities for $(G,p)$ which require further attention; these are the cases arising in part (ii) of Theorem \ref{t:red} and they are handled in Section \ref{s:final}. In each of these cases, the action of $X$ on $V$ is known (up to one of three possibilities if $(G,p) = (E_6,13)$ or $(E_7,19)$) and $X$ stabilizes a non-zero subalgebra of $\la e_{\a} \mid \a \in \Phi^+(G)\ra$. This allows us to reduce to the case where $X$ is contained in a proper parabolic subgroup $P=QL$ of $G$. Let $\pi:P\to P/Q$ be the quotient map. Using $\pi$, we identify $L$ with $P/Q$ and so we may view $\pi(X)$ as a subgroup of $L'$. We may as well assume that $P$ is a minimal parabolic (with respect to containing $X$), so $\pi(X)$ is not contained in a proper parabolic subgroup of $L'$. Now $\pi(x) \in L'$ is a regular unipotent element which is contained in an $A_1$-type subgroup $H$ of $L'$ (this follows by combining Theorem \ref{t:red} with the aforementioned earlier work of Seitz and Testerman \cite{ST1} for classical groups). By inspecting \cite{LT99}, we can determine the action of $H$ on $V$, which must be  compatible with the action of $X$ on $V$ given in Theorem \ref{t:red}. In this way we deduce that $(G,p,L') = (E_6,13,D_5)$ and $(E_7,19,E_6)$ are the only possibilities, and this completes the proof of Theorem \ref{t:main}. 

\section*{Notation}

Our notation is fairly standard. For a simple algebraic group $G$ we write $\Phi(G)$, $\Phi^{+}(G)$ and $\Pi(G) = \{\a_1, \ldots, \a_r\}$ for the set of roots, positive roots and simple roots of $G$, with respect to a fixed Borel subgroup, and we follow Bourbaki \cite{Bou} in labelling the simple roots. We will often denote a root $\a = a_1\a_1+\cdots +a_r\a_r$ by writing $\a = a_1\cdots a_r$. 
If $V$ is a module for a group then ${\rm soc}(V)$ and ${\rm rad}(V)$ denote the socle and radical of $V$, respectively, and we write $V^m$ to denote $V \oplus \cdots \oplus V$ (with $m$ summands). It will be convenient to write $[A_1^{n_1}, \ldots, A_k^{n_k}]$ for a block-diagonal matrix with a block $A_i$ occurring with multiplicity $n_i$. In addition, we will write $J_i$ for a standard (upper triangular) unipotent Jordan block of size $i$. 

\section{Preliminaries}\label{s:prel}

In this section we record some preliminary results that will be needed in the proof of Theorem \ref{t:main}. We start by recalling some well known results from the modular representation theory of the simple groups ${\rm PSL}_{2}(p)$. Our main reference is Alperin \cite{Alperin}.

\subsection{Representation theory}\label{s:rep}

Let $K$ be an algebraically closed field of characteristic $p \geqs 5$, let $X = {\rm PSL}_{2}(p)$ and let $P = \la x \ra \cong Z_p$ be a Sylow $p$-subgroup of $X$. 

The subgroup $P$ of $X$ has exactly $p$ indecomposable $KP$-modules, say $W_i$ for $i=1, \ldots, p$, where $\dim W_i = i$ and $W_p$ is the unique projective indecomposable $KP$-module. The element $x$ has Jordan form $[J_i]$ on $W_i$. In particular, if $W$ is a projective $KP$-module, then $\dim W = ap$ for some $a \geqs 1$, and $x$ has Jordan form $[J_p^a]$ on $W$.

There are precisely $(p+1)/2$ simple $KX$-modules, labelled $V_1, V_3, \ldots, V_p$ in \cite{Alperin}, where $\dim V_i = i$. In particular, every simple $KX$-module is odd-dimensional. Here $V_1$ is the trivial module and $V_p$ is the Steinberg module.  It is easy to see that $x$ has Jordan form $[J_i]$ on $V_i$. By a theorem of Steinberg, each $V_i$ is the restriction of a simple module for the corresponding algebraic group of type $A_1$ (see \cite[Section 13]{St2}), so we can refer to the highest weight of $V_i$ with respect to a maximal torus of the algebraic $A_1$. We identify the weights of this $1$-dimensional torus with the set of integers, and we will 
often write $V_i = L_X(i-1)$ to highlight the highest weight of $V_i$.
 
Similarly, there are precisely $(p+1)/2$ projective indecomposable $KX$-modules, labelled $P_1, P_3, \ldots, P_p$ in \cite{Alperin}, where $P_p = V_p$ is simple and the remainder are reducible. Here $\dim P_1 = \dim P_p = p$ and $\dim P_i = 2p$ for $1<i<p$. The element $x$ has Jordan form $[J_p]$ on $P_1$ and $P_p$, and Jordan form $[J_p^2]$ on the remaining $P_i$. The structure of these modules is described by Alperin \cite[pp.48--49]{Alperin}. In terms of composition factors, 
we have
$$P_1 = V_1|V_{p-2}|V_1$$ 
and
$$P_i = V_i|(V_{p-i+1} \oplus V_{p-i-1})|V_{i}$$
where $1<i<p$ is odd. (Here this notation indicates that ${\rm soc}(P_i) \cong P_i/{\rm rad}(P_i) \cong V_i$ and ${\rm rad}(P_i)/{\rm soc}(P_i) \cong V_{p-i+1} \oplus V_{p-i-1}$.) It will be convenient to define
\begin{equation}\label{e:u}
U = P_1 = L_X(0)|L_X(p-3)|L_X(0)
\end{equation}
and
\begin{equation}\label{e:wi}
W(i) = P_{i+1} = L_X(i)|(L_X(p-i-1)\oplus L_X(p-i-3))|L_X(i)
\end{equation}
for $i \in \{2,4, \ldots, p-3\}$. 
 
The Green correspondence (see \cite[Section 11]{Alperin}) implies that if $V$ is an indecomposable $KX$-module then 
$V|_{P} = W \oplus W'$
where $W$ is projective (or zero) and $W'$ is indecomposable (or zero). In particular, the following lemma holds.

\begin{lem}\label{l:jf}
Let $V$ be an $n$-dimensional indecomposable $KX$-module and write $n = ap+b$, where $a \geqs 0$ and $0 \leqs b < p$. Then $x$ has Jordan form $[J_p^a,J_b]$ on $V$.
\end{lem}

The main result on the structure of indecomposable $KX$-modules is the following theorem. Here we define a \emph{subtuple} of an $n$-tuple $[m_1, \ldots, m_n]$ to be a tuple of the form $[m_i,m_{i+1}, \ldots, m_j]$ for some $1 \leqs i \leqs j \leqs n$. We denote this by writing 
$$[m_i,m_{i+1}, \ldots, m_j] \subseteq [m_1, \ldots, m_n].$$

\begin{thm}\label{t:structure}
Let $V$ be a reducible indecomposable non-projective $KX$-module. Then there exists an integer $\ell \geqs 2$ and a subtuple 
$$[a_1, \ldots, a_{\ell}] \subseteq [1,p-2,3,p-4, \ldots, p-2,1]$$ 
such that 
$${\rm soc}(V) = V_{a_1} \oplus V_{a_3} \oplus \cdots \oplus V_{a_{\ell-\e}},\;\; V/{\rm soc}(V) = V_{a_2} \oplus V_{a_4} \oplus \cdots \oplus V_{a_{\ell-1+\e}}$$
where $\e=1$ if $\ell$ is even, otherwise $\e=0$. 
\end{thm}

\begin{proof}
This follows from the discussion in \cite[Section 3]{Jan}. Also see \cite[Section 7.3]{Craven}. 
 \end{proof}

\begin{cor}\label{c:2step}
Let $V$ be an indecomposable $KX$-module with precisely two composition factors. If ${\rm soc}(V) = L_X(i)$ then $V/{\rm soc}(V) \in \{L_X(p-i-3), L_X(p-i-1)\}$ for some $i \in \{0,2, \ldots, p-3\}$, hence $\dim V = p \pm 1$.  
\end{cor}

\begin{cor}\label{c:bound}
Let $V$ be a reducible indecomposable $KX$-module. Then $\dim V \geqs p-1$. Moreover, if $V$ has at least four composition factors, then $\dim V \geqs 2(p-1)$.
\end{cor}

\subsection{Traces}\label{ss:trace}

As in Section \ref{s:rep}, let $K$ be an algebraically closed field of characteristic $p \geqs 5$ and set $X = {\rm PSL}_{2}(p)$. Let $x_2$ and $x_3$ be representatives of the unique conjugacy classes of elements of order $2$ and $3$ in $X$, respectively (note that $x_i$ is semisimple since $p \geqs 5$). Let $V$ be a $KX$-module and let 
${\rm tr}(V,x_i)$ denote the trace of $x_i$ on $V$. 

\begin{lem}\label{l:1}
If $V = L_X(i)$ then
$${\rm tr}(V,x_2) = (-1)^{i/2},\;\; {\rm tr}(V,x_3) = \left\{\begin{array}{rr}
1 & i \equiv 0 \imod{3} \\
-1 & i \equiv 1 \imod{3} \\
0 & i \equiv 2 \imod{3} 
\end{array}\right.$$
\end{lem}

\begin{proof}
This is a straightforward calculation, using the fact that we can identify $L_X(i)$ with the $i$-th symmetric power ${\rm Sym}^i(M)$, where $M$ is the natural module for ${\rm SL}_{2}(K)$.
 \end{proof}

If $V$ is a $KX$-module with composition factors $M_1, \ldots, M_k$, then 
$${\rm tr}(V,x_i) = \sum_{j=1}^{k}{\rm tr}(M_j,x_i)$$
since the action of $x_i$ is diagonalizable. Therefore, the next two results are immediate corollaries of Lemma \ref{l:1} (here we use the notation $U$ and $W(i)$ defined in (\ref{e:u}) and (\ref{e:wi})).

\begin{lem}\label{l:2}
We have
$${\rm tr}(U,x_2) = \left\{\begin{array}{ll} 1 & p \equiv 1 \imod{4} \\
3 & p \equiv 3 \imod{4} 
\end{array}\right.,\;\; {\rm tr}(U,x_3) = \left\{\begin{array}{ll}
1 & p \equiv 1 \imod{3} \\
2 & p \equiv 2 \imod{3}
\end{array}\right.$$
\end{lem}

\begin{lem}\label{l:3}
We have
$${\rm tr}(W(i),x_2) = \left\{\begin{array}{rr} 
2 & i \equiv 0 \imod{4} \\
-2 & i \equiv 2 \imod{4} 
\end{array}\right.$$
and
$${\rm tr}(W(i),x_3) = \left\{\begin{array}{ll}
\left\{\begin{array}{ll}
2 & p \equiv 1 \imod{3} \\
1 & p \equiv 2 \imod{3}
\end{array}\right. & i \equiv 0 \imod{3} \\
& \\
\left\{\begin{array}{ll}
-1 & p \equiv 1 \imod{3} \\
-2 & p \equiv 2 \imod{3}
\end{array}\right. & i \equiv 1 \imod{3} \\
& \\
\left\{\begin{array}{rr}
-1 & p \equiv 1 \imod{3} \\
1 & p \equiv 2 \imod{3}
\end{array}\right. & i \equiv 2 \imod{3} 
\end{array}\right.$$
\end{lem}

Let $G$ be a simple algebraic group over $K$ of adjoint type, let $V$ be a $KG$-module and let $m$ be a positive integer. Define 
$$\mathcal{T}_{m}(G,V) = \{{\rm tr}(V,x) \mid \mbox{$x \in G$ has order $m$} \}.$$ 
Recall that the \emph{adjoint module} for $G$ is the Lie algebra ${\rm Lie}(G)$, on which $G$ acts via the adjoint representation.

\begin{prop}\label{p:trace}
Let $G$ be a simple exceptional algebraic group of adjoint type over an algebraically closed field of characteristic $p \geqs 5$. Let $V = {\rm Lie}(G)$ be the adjoint module. Then $\mathcal{T}_{m}(G,V)$ is recorded in Table 
$\ref{t:trace}$ for $m \in \{2,3\}$.
\end{prop}

\begin{table}
$$\begin{array}{lll} \hline
G & \mathcal{T}_{2}(G,V) & \mathcal{T}_{3}(G,V) \\ \hline
G_2 & -2 & -1, 5 \\
F_4 & -4, 20 & -2,7 \\
E_6 & -2, 14 & -3, 3, 6, 15, 30 \\
E_7 & -7,5,25 & -2, 7, 34, 52 \\ 
E_8 & -8, 24 & -4, 5, 14, 77 \\ \hline
\end{array}$$
\caption{Traces of elements of order $2$ and $3$ on the adjoint module}
\label{t:trace}
\end{table}

\begin{proof}
This follows by inspecting the dimensions of the centralizers of elements of order $m$ in $G$ (see \cite[Tables 4.3.1 and 4.7.1]{GLS}), using the fact that
$$\dim C_{V}(g) = \dim C_G(g)$$
for every semisimple element $g \in G$ (see \cite[Section 1.14]{Carter}, for example). 

For instance, if $g \in G = E_8$ has order $3$ and $C_G(g) = A_8$, then $\dim C_V(g) = 80$ and the self-duality of $V$ implies that the action of $g$ on $V$ is given by the diagonal matrix $[I_{80}, \omega I_{84}, \omega^2 I_{84}]$, up to conjugacy, where $\omega$ is a primitive cube root of unity. Therefore, ${\rm tr}(V,g) = -4$.
 \end{proof}

\begin{remk}\label{r:eeig}
Suppose $X = {\rm PSL}_{2}(p)$ is contained in $G=E_6$, where $p \geqs 5$ and $G$ is adjoint. Write $G = \hat{G}/S$ and $X = \hat{X}/S$, where $\hat{G}$ is the simply connected group of type $E_6$ and $S = Z_3$ is the centre of $\hat{G}$. Now $X$ has Schur multiplier $Z_2$, which implies that $\hat{X} = Z_3 \times X$. Therefore, every element $y \in X$ of order $3$ lifts to an element in $\hat{G}$ of order $3$. In particular, if $y \in X$ has order $3$ then $C_G(y)^0 = A_5T_1$, $D_4T_2$ or $A_2^3$ (see \cite[Table 4.7.1]{GLS}), whence ${\rm tr}(V,y) \in \{-3,6,15\}$ with respect to $V = {\rm Lie}(G)$.
\end{remk}

\begin{remk}\label{r:AJL}
In a few cases it is helpful to know the eigenvalue multiplicities on $V$ of elements in $G$ of order $m>3$ for certain values of $m$; the relevant cases are the following:
$$(G,m) \in \{(F_4,7), (E_6,7), (E_7,5), (E_8,19)\}.$$
It is straightforward to obtain this information with the aid of {\sc Magma} \cite{magma}, using an algorithm of Litterick (see \cite[Section 3.3.1]{AJL}), which is heavily based on work of Moody and Patera \cite{MP}. We thank Dr. Litterick for his assistance with these computations. 
\end{remk}

\subsection{$A_1$-type subgroups}\label{ss:a1}

Let $G$ be a simple algebraic group and recall that $p$ is a \emph{good} prime for $G$ if $p>2$ in types $B,C$ and $D$, $p>3$ for $G_2,F_4,E_6$ and $E_7$, and $p>5$ when $G$ is of type $E_8$ (all primes are good in type $A$). 

\begin{prop}\label{p:0}
Let $G$ be a simple algebraic group of adjoint type over an algebraically closed field of good characteristic $p>0$. Let $x \in G$ be an element of order $p$. 
\begin{itemize}\addtolength{\itemsep}{0.2\baselineskip}
\item[{\rm (i)}] There is an $A_1$-type subgroup of $G$ containing $x$.
\item[{\rm (ii)}] If $x$ is regular then the subgroup in {\rm (i)} is unique up to $G$-conjugacy. 
\end{itemize}
\end{prop}

\begin{proof}
Part (i) follows from the main theorem of \cite{T}. Part (ii), for $G$ exceptional, follows from \cite[Theorem 4]{LT99}. Now assume $G$ is classical and let $H$ be an $A_1$-type subgroup of $G$ containing $x$. Let $V$ be the natural module for $G$. By \cite[Theorem 1.2]{TZ}, $H$ is not contained in a proper parabolic subgroup of $G$. In particular, if $G$ is of type $A, B$ or $C$ then $H$ acts irreducibly and tensor indecomposably (see \cite[Proposition 2.3]{TZ}) on $V$  and the conjugacy statement follows from representation theory. 

Finally, let us assume $G = D_r$ (with $r \geqs 4$). We claim that $H<L<G$, where $L=B_{r-1}$ is the stabilizer of a non-singular $1$-space. The result then follows since $H$ is unique in $L$ up to $L$-conjugacy, and $L$ itself is unique up to $G$-conjugacy. To justify the claim, first observe that $x$ has Jordan form $[J_{2r-1},J_1]$ on $V$, using \cite[Lemma 1.2(ii)]{SS} and the fact that $x$ has order $p$, so $p\ne 2$. If $H$ acts irreducibly on $V$ then the Jordan form of $x$ implies that $H$ is tensor decomposable, but this is incompatible with \cite[Lemma 1.5]{SS}. Therefore, $H$ acts reducibly on $V$ and we complete the argument by applying \cite[Lemma 2.2]{LT}.
 \end{proof}

\begin{prop}\label{p:11}
Let $G$ be a simple exceptional algebraic group of adjoint type and let $x \in G$ be a regular unipotent element such that  
$$x \in X = {\rm PSL}_{2}(p)<A<G,$$ 
where $A$ is an $A_1$-type subgroup. Then the action of $X$ on the adjoint module $V = {\rm Lie}(G)$ is given in Table $\ref{t:dec}$.
\end{prop}

\begin{proof}
A precise description of $V|_{A}$ as a tilting module is given in \cite[Table 10.1]{LS04} (for $G=E_6$ we may assume that $A<F_4<G$ so the action of $A$ on $V$ can be deduced from the actions of $A$ on ${\rm Lie}(F_4)$ and the minimal module for $F_4$ (see \cite[Table 10.2]{LS04})). Following \cite{LS04}, we write $T(\l;\mu; \ldots)$ for a tilting module having the same composition factors as the direct sum of Weyl modules for $A$ with highest weights $\l,\mu, \ldots$ In terms of this notation, we get
$$V|_{A} = \left\{\begin{array}{ll}
T(10;2) & G=G_2 \\
T(22;14;10;2) & G=F_4 \\
T(22;16;14;10;8;2) & G=E_6 \\
T(34;26;22;18;14;10;2) & G=E_7 \\
T(58;46;38;34;26;22;14;2) & G=E_8 
\end{array}\right.$$

As explained at the start of \cite[Section 10]{LS04}, we can express $T(\l;\mu; \ldots)$ as a direct sum of indecomposable tilting modules of the form $T(c)$, where the highest weight $c$ is at most $2p-2$. For example, suppose $G=F_4$ and $p=19$, so $V|_A = T(22;14;10;2)$ as above. The highest weight is $22$, so one summand is $T(22)$, which is a uniserial module of shape $14|22|14$ (see \cite[Lemma 2.3]{Seitz}). The highest weight not already accounted for is $10$, so $T(10) = L_A(10)$ is a summand and we deduce that $V|_{A} = T(22) \oplus L_A(10) \oplus L_A(2)$ and thus
$$V|_{X} = T(22)|_{X} \oplus L_X(10) \oplus L_X(2).$$
By \cite[Lemma 2.3]{Seitz}, $T(22)|_{X}$ is a projective indecomposable $KX$-module of dimension $2p=38$, so $T(22)|_{X} = W(i)$ for some $i$. By comparing socles, it follows that $i=14$ and thus 
$$V|_{X} = W(14) \oplus L_X(10) \oplus L_X(2)$$
as recorded in Table \ref{t:dec}. The other cases are entirely similar and we omit the details.
 \end{proof}

\begin{table}
$$\begin{array}{lll} \hline
G & p & V|_{X} \\ \hline
G_2 & p \geqs 11 & L_X(10) \oplus L_X(2) \\
& p = 7 & W(2) \\
F_4 & p \geqs 23 & L_X(22) \oplus L_X(14) \oplus L_X(10) \oplus L_X(2)  \\
&  p = 19 & W(14) \oplus L_X(10) \oplus L_X(2)  \\
&  p = 17 & W(10) \oplus L_X(14) \oplus L_X(2)  \\
&  p = 13 & W(10) \oplus W(2)  \\
E_6 & p \geqs 23 & L_X(22) \oplus L_X(16) \oplus L_X(14) \oplus L_X(10) \oplus L_X(8) \oplus L_X(2)  \\
& p = 19 & W(14) \oplus L_X(16) \oplus L_X(10) \oplus L_X(8) \oplus L_X(2)  \\
&  p = 17 & W(10) \oplus L_X(16) \oplus L_X(14) \oplus L_X(8) \oplus L_X(2)  \\
&  p = 13 & W(10) \oplus W(8) \oplus W(2)  \\
E_7 & p \geqs 37 & L_X(34) \oplus L_X(26) \oplus L_X(22) \oplus L_X(18) \oplus L_X(14) \oplus L_X(10) \oplus L_X(2)  \\ 
& p = 31 & W(26) \oplus L_X(22) \oplus L_X(18) \oplus L_X(14) \oplus L_X(10) \oplus L_X(2)  \\
& p = 29 & W(22) \oplus L_X(26) \oplus L_X(18) \oplus L_X(14) \oplus L_X(10) \oplus L_X(2)  \\
& p = 23 & W(18) \oplus W(10) \oplus L_X(22) \oplus L_X(14) \oplus L_X(2)  \\
 & p = 19 & W(14) \oplus W(10) \oplus W(2) \oplus L_X(18)  \\
E_8 & p \geqs 59 & L_X(58) \oplus L_X(46) \oplus L_X(38) \oplus L_X(34)  \oplus L_X(26)  \oplus L_X(22) \\
& &  \oplus \,  L_X(14) \oplus L_X(2)  \\ 
& p = 53 & W(46) \oplus L_X(38) \oplus L_X(34)  \oplus L_X(26)  \oplus L_X(22)\oplus L_X(14) \oplus L_X(2)    \\
 & p = 47 & W(34) \oplus L_X(46) \oplus L_X(38)  \oplus L_X(26)  \oplus L_X(22)  \oplus L_X(14) \oplus L_X(2) \\
 & p = 43  & W(38) \oplus W(26) \oplus L_X(34) \oplus L_X(22)  \oplus L_X(14) \oplus L_X(2)  \\
 & p = 41 & W(34) \oplus W(22) \oplus L_X(38) \oplus L_X(26)  \oplus L_X(14) \oplus L_X(2)   \\
 & p = 37 & W(34) \oplus W(26) \oplus W(14) \oplus L_X(22) \oplus L_X(2)   \\
 & p = 31 & W(26) \oplus W(22) \oplus W(14) \oplus W(2)  \\ \hline
\end{array}$$
\caption{The action of $X$ on $V={\rm Lie}(G)$ in Proposition \ref{p:11}}
\label{t:dec}
\end{table}

For the remainder of this section, let $G$ be a simple exceptional algebraic group of adjoint type over an algebraically closed field $K$ of 
characteristic $p>0$, and let $r$ and $h = h(G)$ be the rank and Coxeter number of $G$, respectively. We will assume $G$ contains a regular unipotent element of order $p$, which means that 
\begin{equation}\label{e:pbdh}
p \geqs h.
\end{equation}
We need to recall the construction of $A_1$-type subgroups of $G$ containing regular unipotent elements, following the treatment in \cite{Serre, T92, T}.

First we need some new notation. Let ${\mathcal L}_{\mathbb C}$ be a simple Lie algebra over $\mathbb{C}$ of type 
$\Phi(G)$. Fix a Chevalley basis 
$$\mathcal{B} = \{e_{\a}, f_{\a}, h_{\gamma} \mid \a \in \Phi^{+}(G), \, \gamma \in \Pi(G) \}$$
of ${\mathcal L}_{\mathbb C}$ and write $z_i = z_{\a_i}$ for $z \in \{e,f,h\}$ and $\Pi(G) = \{\a_1, \ldots, \a_r\}$. 
It will be convenient to define $f_{\a} = e_{-\a}$ for each $\a \in \Phi^+(G)$. Let ${\mathcal L}_\Z$ be the $\Z$-span of $\mathcal{B}$ and set ${\mathcal L}_K = \mathcal{L}_{\Z} \otimes_{\Z} K$.  (By abuse of notation, we also write $e_\alpha, f_\alpha, e_i, f_i, h_i$ for the elements
$e_\alpha\otimes 1$, $f_\alpha\otimes 1$, etc., in ${\mathcal L}_K$.) Fix a root $\a \in \Phi(G)$. As in the familiar Chevalley construction, we have
$$({\rm ad}(e_\alpha)^j/j!)({\mathcal L}_\Z)\subseteq {\mathcal L}_\Z$$
for all $j\geqs 0$, and this allows us to construct the element 
\[
{\rm exp}({\rm ad}({\bf x}e_\alpha))\in {\rm GL}_{\dim G}(\Z[{\bf x}]),
\]
where ${\bf x}$ is an indeterminate. Passing to $K$, we obtain a $1$-dimensional unipotent subgroup
$$U_{\a} = \{ {\rm exp}({\rm ad}(\gamma e_\alpha))\mid\gamma\in K\}\leqs {\rm Aut}({\mathcal L}_K)^0 = G$$ 
(see \cite[Proposition 4.4.2]{Carter0}). Note that $G = \la U_{\a} \mid \a \in \Phi(G)\ra$. 

Given the lower bound on $p$ in (\ref{e:pbdh}), we can make a similar construction for more general elements of 
${\mathcal L}_\Z$. To do this, let ${\mathcal L}_{\Z_{(p)}}$ be the $\Z_{(p)}$-span of $\mathcal{B}$, where $\Z_{(p)}$ is the localization of $\Z$ at the prime ideal $(p) = p\Z$, so that 
${\mathcal L}_K = {\mathcal L}_{\Z_{(p)}}\otimes_{\Z_{(p)}} K$.
By \cite[Proposition 1.5]{T} we have 
$$\left({\rm ad}(e)^j/j!\right)({\mathcal L}_{\Z_{(p)}}) \subseteq {\mathcal L}_{\Z_{(p)}}$$
for all $e\in\sum_{\alpha\in\Phi^+(G)}\Z e_\alpha$ and all $j\geqs 0$. Then as in the Chevalley construction, for any non-zero element $y$ in $\sum_{\alpha\in\Phi^+(G)}\Z e_\alpha$ or $\sum_{\alpha\in\Phi^+(G)}\Z f_\alpha$, we can produce
$$x_y({\bf x}) = {\rm exp}({\rm ad}({\bf x} y))\in{\rm GL}_{\dim G}(\Z_{(p)}[{\bf x}]).$$
In particular, by passing to $K$, we define 
\begin{equation}\label{e:xyy}
U_y = \{x_y(\gamma)={\rm exp}({\rm ad}(\gamma y)) \mid \gamma\in K \} \subseteq {\rm Aut}({\mathcal L}_K)^0 = G.
\end{equation}

We will use this general set-up to  construct certain $A_1$-type subgroups of our group $G$, following \cite{T92,T}. In order to state the main result (Proposition \ref{p:a1key} below), recall that an ordered triple of elements $(e,h,f)$ 
chosen from $\mathcal L_K$ (or from ${\mathcal L}_{\Z}$) is an \emph{$\mathfrak{sl}_2$-triple} if the elements satisfy the commutation relations between the standard generators of the Lie algebra $\mathfrak{sl}_2$, namely
$$[h,e]=2e,\; [h,f] = -2f, \; [e,f] = h.$$ 

We have the following result (in part (iii), we use the notation $x_y(\gamma)$ from (\ref{e:xyy})).

\begin{prop}\label{p:a1key}
Suppose $p\geqs h(G)$ and $(e,h,f)$ is an $\mathfrak{sl}_2$-triple of ${\mathcal L}_\Z$, with $e=\sum_{i=1}^re_i$ and $f\in\sum_{i=1}^r\Z f_i$. 
Then the following hold:
\begin{itemize}\addtolength{\itemsep}{0.2\baselineskip}
\item[{\rm (i)}]  $U_e$ and $U_f$ are $1$-dimensional subgroups of $G$.
\item[{\rm (ii)}] $A = \langle U_e,U_f\rangle$ is an $A_1$-type subgroup of $G$.
\item[{\rm (iii)}] $T =\{t(c) \mid c\in K^\times\}$ is a maximal torus of $\langle U_e,U_f\rangle$, where 
$$t(c) = x_e(c)x_f(-c^{-1})x_e(c)x_e(-1)x_f(1)x_e(-1),$$
and the map $t:\mathbb{G}_m \to T$ is a morphism of algebraic groups.
\item[{\rm (iv)}] The action of $T$ on the basis  $\{{\bar v} = v\otimes 1 \mid v \in \mathcal{B}\}$ of $\mathcal{L}_K$ is given by 
$$t(c)\cdot {\bar e_\alpha} = c^{\alpha(h)}{\bar e_\alpha},\;\; t(c)\cdot {\bar h}_i = {\bar h}_i$$
for all $\alpha\in\Phi(G)$, $1 \leqs i \leqs r$. Moreover, $\alpha_i(h) = 2$ for all $1\leqs i\leqs r$.
\item[{\rm (v)}] $T$ normalizes $U_e$ and $U_f$.
\item[{\rm (vi)}] $U_e$ contains a regular unipotent element of $G$.
\end{itemize}
\end{prop} 

\begin{proof}
This follows by combining Lemmas 1 and 2 in \cite{T92} with Lemma 1.2 in \cite{T}.
 \end{proof}

The following result will play an important role in the proof of Theorem \ref{t:main}.

\begin{prop}\label{p:new1}
Suppose $p\geqs h(G)$ and $(e,h,f)$ is an $\mathfrak{sl}_2$-triple of ${\mathcal L}_\Z$, with $e=\sum_{i=1}^re_i$ and $f\in\sum_{i=1}^r\Z f_i$. Let $W$ be the $3$-dimensional subalgebra of ${\mathcal L}_K$ generated by $\{e,f\}$ and let $H$ be the stabilizer of $W$ in $G$. Then $H$ is an $A_1$-type subgroup of $G$.
\end{prop}

\begin{proof} 
Let $A$ be the $A_1$-type subgroup of $G$ constructed in Proposition \ref{p:a1key}(ii). Note that $A$ contains a regular unipotent element and it clearly stabilizes $W$ by construction, so $A \leqs H$. Let $M_1$ be a maximal closed positive dimensional subgroup of $G$ with $H \leqs M_1$. By the main theorem of \cite{TZ}, $A$ is not contained in a proper parabolic subgroup of $G$, so Borel-Tits \cite[Corollary 3.9]{BT} (also see Weisfeiler \cite{Weis}) implies that $M_1$ is reductive. By \cite[Theorem A]{SS}, either $M_1$ is an $A_1$-type subgroup (and thus $A=H=M_1$), or $G=E_6$ and $M_1 = F_4$. In the latter case, $H \ne M_1$ since $M_1$ does not stabilize a $3$-dimensional subspace of ${\mathcal L}_K$, so let $M_2$ be a maximal closed positive dimensional subgroup of $M_1$ with $H \leqs M_2$. As above, $M_2$ is reductive and by applying \cite[Theorem A]{SS} once again, we conclude that $A=H=M_2$.
 \end{proof}

We would like to be able to use Proposition \ref{p:new1} to identify the stabilizers of other $\mathfrak{sl}_2$-subalgebras of $\mathcal{L}_K$. With this aim in mind, we present Proposition \ref{p:new2} below. In order to state this result, we need some additional notation. 

Suppose we have an $\mathfrak{sl}_2$-triple $(e,h,f)$ as in Proposition \ref{p:a1key}. Let $T$ be the $1$-dimensional torus constructed in part (iii) of the proposition. Let $\alpha_0 \in \Phi(G)$ be the highest root and recall that $h(G) = {\rm ht}(\a_0)+1$, where ${\rm ht}:\Phi(G)\to \mathbb N$ is the familiar height function (that is, if $\a = \sum_{i}a_i\a_i$ then ${\rm ht}(\a) = \sum_{i}a_i$). Then
$$\{2i \mid -{\rm ht}(\alpha_0) \leqs i \leqs {\rm ht}(\alpha_0)\}$$
is the set of weights of $T$ on both ${\mathcal L}_\Z$ and ${\mathcal L}_K$. For each $T$-weight $m$, write 
$({\mathcal L}_\Z)_m$ for the corresponding $T$-weight space and similarly for ${\mathcal L}_K$. In both the statement and proof of the following result, we use the notation $\bar{a} = a \otimes 1\in {\mathcal L}_K$ for $a \in {\mathcal L}_\Z$.

\begin{prop} \label{p:new2}
Suppose $p \geqs h(G)$ and $(e,h,f)$ is an $\mathfrak{sl}_2$-triple of ${\mathcal L}_\Z$, with $e=\sum_{i=1}^r e_i$ 
and $f\in\sum_{i=1}^r\Z f_i$. Suppose $y\in ({\mathcal L}_\Z)_{p-1}\cap C_{{\mathcal L}_\Z}(e)$ and
$z\in({\mathcal L}_\Z)_{p-3}$ are chosen so that 
\begin{itemize}\addtolength{\itemsep}{0.2\baselineskip}
\item[{\rm (i)}] $[y,z] = 0$ in ${\mathcal L}_\Z$; and
\item[{\rm (ii)}] $({\bar e},{\bar{h}+\gamma\bar{ y}},\bar{f}+\gamma\bar{ z})$ is an $\mathfrak{sl}_2$-triple 
in ${\mathcal L}_K$ for some $\gamma\in K$.
\end{itemize}
Then there exists $g\in C_G({\bar e})$ such that $g \cdot {\bar h}={\bar{h}+\gamma \bar{y}}$ and 
$g \cdot {\bar f} = {\bar{f}+\gamma \bar{z}}$ in ${\mathcal L_K}$. Moreover, the stabilizer in $G$ of the 
subalgebra $W$ of $\mathcal{L}_K$ generated by $\{{\bar e}, {\bar{f}+\gamma \bar{z}}\}$ is an $A_1$-type 
subgroup. 
\end{prop} 

\begin{proof} 
First observe that $y \in \sum_{\a \in \Phi^+(G)}\Z e_{\a}$ since $y \in (\mathcal{L}_{\Z})_{p-1}$, so we can take $g = x_y(\gamma)\in G$ as in (\ref{e:xyy}). Note that $g \in C_G({\bar e})$ since $y\in C_{{\mathcal L}_\Z}(e)$. Now $y$ is an eigenvector for ${\rm ad}(h)$ (since $y$ is a $T$-weight vector), so $[y,[y,h]]=0$ and thus
$$g \cdot {\bar h} = {\bar h}+\gamma [{\bar y}, {\bar h}] = {\bar h}-\gamma[{\bar h},{\bar y}] = 
{\bar h}-\gamma(p-1){\bar y} = {\bar h}+\gamma{\bar y}.$$ 
The maximum $T$-weight in ${\mathcal L_\Z}$ is $2{\rm ht}(\alpha_0)$, which is at most $2(p-1)$ since $p \geqs h(G)>{\rm ht}(\alpha_0)$, so ${\rm ad}(y)^i(f)\in ({\mathcal L}_\Z)_{i(p-1)-2} = 0$ for all $i\geqs 3$ and thus
$$g \cdot {\bar f} = {\bar f}+\gamma[{\bar y},{\bar f}] + \frac{1}{2}\gamma^2[{\bar y},[{\bar y},{\bar f}]].$$
In addition, since $[y,z]=0$ and $z \in (\mathcal{L}_{\Z})_{p-3}$, we have 
\begin{equation}\label{e:dis}
[h+y,f+z] = -2f+(p-3)z+[y,f].
\end{equation}
The $\mathfrak{sl}_2$ commutation relations imply that $[{\bar h}+{\bar y},{\bar f}+{\bar z}] = -2({\bar f}+{\bar z})$, 
which is equal to $-2{\bar f}-3{\bar z}+[{\bar y},{\bar f}]$ by (\ref{e:dis}). Therefore, 
$[{\bar y},{\bar f}] = {\bar z}$ and thus $[{\bar y},[{\bar y},{\bar f}]] = 0$. We conclude that $g \cdot {\bar h} = {\bar{h}+\gamma\bar{y}}$ and 
$g \cdot {\bar f} = {\bar{f}+\gamma\bar{z}}$ in ${\mathcal L}_K$, as required. The final statement concerning the stabilizer of $W$ follows immediately from Proposition \ref{p:new1}.
 \end{proof}

\subsection{Exponentiation}\label{ss:exp}

In this section we turn to a different notion of ``exponentiation'',  following Seitz \cite{Seitz}. As before, let $G$ be a simple exceptional algebraic group of adjoint type over an algebraically closed field $K$ of characteristic $p>0$ and let $r$ and $h$ denote the rank and Coxeter number of $G$, respectively. 
Let $U = \la U_{\a} \mid \a \in \Phi^{+}(G)\ra$ be the unipotent radical of a fixed Borel subgroup $B$ of $G$ corresponding to our choice of base 
$\Pi(G) = \{\a_1, \ldots, \a_r\}$, where the root subgroup $U_{\a}$ is defined as in (\ref{e:xyy}). As explained in \cite[Section 5]{Seitz}, we may view ${\rm Lie}(U)$ as an algebraic group via the Hausdorff formula. Set $V = {\rm Lie}(G)$.

We start by recalling \cite[Proposition 5.3]{Seitz}. 

\begin{prop}\label{l:exp00}
Suppose $p \geqs h$. Then there exists a unique isomorphism of algebraic groups
\begin{equation}\label{e:exp}
\theta : {\rm Lie}(U) \to U
\end{equation}
whose tangent map is the identity and which is $B$-equivariant; that is, $\theta(b \cdot n) = b\theta(n)b^{-1}$ for all $n \in{\rm Lie}(U)$, $b \in B$.
\end{prop}

Suppose $G$ contains a regular unipotent element $x$ of order $p$, so $p \geqs h$ and we are in a position to use Proposition \ref{l:exp00} to study the structure of $C_G(x)$. Replacing $x$ by a suitable conjugate, we may assume that 
$$x = x_{e}(1) = {\rm exp}({\rm ad}(e)),$$ 
where $e = \sum_{i=1}^re_i$. As in Proposition \ref{p:a1key}, let $A$ be an $A_1$-type subgroup of $G$ containing $x$, and let $T = \{t(c) \mid c \in K^{\times}\}$ be the given maximal torus of $A$. Without loss of generality, we may assume that $T$ is contained in the Borel subgroup $B$ defined above. From the description of the action of $A$ on $V = {\rm Lie}(G)$ in the proof of Proposition \ref{p:11}, it follows that $t(c)$ acts on the $1$-eigenspace $C_V(x) = {\rm Lie}(C_G(x))$ as 
\begin{equation}\label{e:dii}
{\rm diag}(c^{d_1},\ldots,c^{d_r}),
\end{equation}
where the $d_i$ are recorded in Table \ref{tab:di} (we label the $d_i$ so that they form a decreasing sequence). 

\begin{table}
$$\begin{array}{ll} \hline
G & d_i \\ \hline
G_2 & 10, 2 \\
F_4 & 22, 14, 10, 2 \\
E_6 &  22, 16, 14, 10, 8, 2 \\
E_7 &  34, 26, 22, 18, 14, 10, 2 \\
E_8 &  58, 46, 38, 34, 26, 22, 14, 2 \\ \hline
\end{array}$$
\caption{The integers $d_1, \ldots, d_r$ in (\ref{e:dii})}
\label{tab:di}
\end{table}

\begin{prop}\label{l:exp1}
Let $x = x_e(1) \in G$ be a regular unipotent element of order $p$, where $e = \sum_{i=1}^{r}e_i$, and let $T = \{t(c) \mid c \in K^{\times}\}$ be the torus constructed in Proposition \ref{p:a1key}. Then there exist connected $1$-dimensional unipotent subgroups $X_{i} = \{x_{i}(\gamma) \mid \gamma \in K\}$ such that the following hold:
\begin{itemize}\addtolength{\itemsep}{0.2\baselineskip}
\item[{\rm (i)}]   $C_G(x) = \la X_{i} \mid 1 \leqs i \leqs r \ra$. In particular, each $z \in C_G(x)$ can be written as a commuting product of the form $z = \prod_{i=1}^{r}x_{i}(\gamma_i)$ for some $\gamma_i \in K$. 
\item[{\rm (ii)}] We have  
\begin{equation}\label{e:tac}
t(c)x_i(\gamma)t(c)^{-1} = x_i(c^{d_i}\gamma)
\end{equation}
for all $c \in K^{\times}$, $\gamma \in K$, $1 \leqs i \leqs r$.
\end{itemize}
\end{prop}

\begin{proof}
First note that $p\geqs h$ since $x$ has order $p$. As above, let 
\[
U = \la U_{\a} \mid \a \in \Phi^{+}(G)\ra
\]
be the unipotent radical of a Borel subgroup $B$ of $G$ and note that $x \in U$ and $T \leqs B$. Moreover, we have $C_G(x) \leqs U$ and thus $C_V(x)  = {\rm Lie}(C_G(x))\subseteq {\rm Lie}(U)$. Choose $v_r \in {\rm Lie}(U)$ such that $\theta(v_r) = x$, where $\theta$ is the map in Proposition \ref{l:exp00}. Extend to a basis $\{v_1, \ldots, v_r\}$ of the $1$-eigenspace $C_V(x)$, where $t(c) \cdot v_i = c^{d_i}v_i$ for each $i$, and construct the corresponding connected $1$-dimensional unipotent subgroups
$$X_i = \{x_i(\gamma)  = \theta(\gamma v_i) \mid \gamma \in K\} \leqs G.$$
Recall that $C_G(x)$ is abelian, so $C_V(x) = {\rm Lie}(C_G(x))$ is an abelian subalgebra and the proof of \cite[Proposition 5.4]{Seitz} implies that each $X_i$ is contained in $C_G(x)$. Therefore, $H = \la X_i \mid 1 \leqs i \leqs r\ra$ is a closed connected unipotent subgroup of $C_G(x)$. Moreover, $v_i \in {\rm Lie}(X_i)$ for each $i$, so $\dim H \geqs r$ and thus $H = C_G(x)$ (note that $C_G(x)$ is connected since $G$ is adjoint). Part (i) now follows since $C_G(x)$ is abelian. Finally, part (ii) follows from the $B$-equivariance of $\theta$ (see Proposition \ref{l:exp00}).
 \end{proof}

\begin{prop}\label{l:bt}
Let $U$ be the unipotent radical of a Borel subgroup $B$ of $G$, let $W$ be a proper non-zero subalgebra of ${\rm Lie}(U)$ and let $H$ be the stabilizer of $W$ in $G$.  Assume $H$ contains a regular unipotent element of $G$ of order $p$. Then either 
\begin{itemize}\addtolength{\itemsep}{0.2\baselineskip}
\item[{\rm (i)}] $H$ is contained in a proper parabolic subgroup of $G$; or
\item[{\rm (ii)}] $H$ is contained in an $A_1$-type subgroup of $G$.  
\end{itemize}
\end{prop}

\begin{proof}
Since $p \geqs h$, we can consider the isomorphism $\theta: {\rm Lie}(U) \to U$ in (\ref{e:exp}). Let $Z=Z(W)$ be the centre of $W$, which is a non-zero abelian subalgebra of $W$ stabilized by $H$. We claim that $\theta(Z) \leqs H$. To see this, let $z \in Z$, $w \in W$ and note that $\theta(z)$ and $\theta(w)$ commute since $[z,w]=0$ in ${\rm Lie}(U)$ (see the proof of \cite[Proposition 5.4]{Seitz}). The $B$-equivariance of $\theta$ implies that
$$\theta(\theta(z)\cdot w) = \theta(z)\theta(w)\theta(z)^{-1} = \theta(w),$$
so $\theta(z) \cdot w = w$ and the claim follows. Therefore, $H$ is a positive dimensional subgroup of $G$ containing a regular unipotent element.

To complete the argument, we proceed as in the proof of Proposition \ref{p:new1}, using Borel-Tits \cite[Corollary 3.9]{BT}. Let us assume $H$ is not contained in a proper parabolic subgroup of $G$. Then $H \leqs M_1$, where $M_1$ is a maximal closed reductive positive dimensional subgroup of $G$. By the main theorem of \cite{SS}, either $M_1$ is an $A_1$-type subgroup, or $G = E_6$ and $M_1=F_4$, so we may assume that we are in the latter situation. Suppose $H=M_1$. Since 
$$V|_{M_1} = {\rm Lie}(M_1) \oplus V_{26},$$
where $V_{26}$ is the minimal module for $M_1$, it follows that $W=V_{26}$ is the only possibility. But $V_{26}$ must contain non-zero elements in the Lie algebra of a maximal torus of $G$ (just by comparing dimensions) and this is a contradiction. Therefore $H$ is a proper subgroup of $M_1$ and thus $H \leqs M_2$ for some maximal closed reductive subgroup $M_2$ of $M_1$. By a further application of \cite{SS} we conclude that $H$ is contained in an $A_1$-type subgroup of $G$. 
 \end{proof}

\subsection{Methods}\label{ss:me}

In this section we discuss the proof of Theorem \ref{t:main}, highlighting the main steps and ideas. 

Let $G$ be a simple exceptional algebraic group of adjoint type defined over an algebraically closed field $K$ of characteristic $p>0$. Let $r$ be the rank of $G$ and let $V = {\rm Lie}(G)$ be the adjoint module. Suppose $x \in X = {\rm PSL}_{2}(p)<G$ is a regular unipotent element of $G$, so $p \geqs h$ where $h$ is the Coxeter number of $G$. The embedding of $X$ in $G$ corresponds to an abstract homomorphism $\varphi:{\rm SL}_2(p)\to G$ with kernel $Z = Z({\rm SL}_{2}(p))$ and image $X$. 

As before, let $\mathcal{L}_{\mathbb{C}}$ be a simple Lie algebra over $\mathbb{C}$ of type $\Phi(G)$ and fix a Chevalley basis 
\begin{equation}\label{e:chev}
\mathcal{B} = \{e_{\a}, f_{\a}, h_{\gamma} \mid \a \in \Phi^{+}(G), \, \gamma \in \Pi(G) \}.
\end{equation}
Since $p \geqs h$, we can view $\mathcal{B}$ as a basis for $V$, where $e_{\a}, f_{\a}$ are in the appropriate root spaces with respect to the Cartan subalgebra spanned by the $h_{\gamma}$. 
It will be convenient to write $z_i = z_{\a_i}$ for $z \in \{e,f,h\}$ and $\Pi(G) = \{\a_1, \ldots, \a_r\}$. 

Set $e = \sum_{i=1}^re_i$ and let $(e,h,f)$ be an $\mathfrak{sl}_2$-triple as in Proposition \ref{p:a1key}. Let $A$ be the corresponding $A_1$-type subgroup of $G$ constructed in Proposition \ref{p:a1key}, with maximal torus $T = \{t(c) \mid c \in K^{\times}\}$ and associated morphism $t:\mathbb{G}_m \to T$. By replacing $X$ by a suitable $G$-conjugate, we may assume that $x = {\rm exp}({\rm ad}(e)) \in A$. 
Let $v:\mathbb{G}_a \to A$ be a morphism of algebraic groups such that 
$$t(c)v(\gamma)t(c)^{-1} = v(c^2\gamma)$$ 
for all $c\in K^{\times}$, $\gamma \in K$. We may assume $v: \mathbb{G}_a \to {\rm im}(v)$ is an isomorphism of algebraic groups.  

Consider the elements
\begin{equation}\label{e:mats}
u=\left(\begin{array}{cc} 1&1\\ 0&1\end{array}\right), \;\; s = \left(\begin{array}{cc} \xi&0 \\ 0&\xi^{-1}\end{array}\right)
\end{equation}
in ${\rm SL}_{2}(p)$, where $\mathbb{F}_{p}^{\times} = \la \xi \ra = \{1, \ldots, p-1\}$. 
Without loss of generality, we may assume that $x=\varphi(u) = v(1)$ so $\varphi(sus^{-1}) = \varphi(u^m)  = v(m)$ with $m = \xi^2$.  
Then $t(\xi)xt(\xi)^{-1} = \varphi(s) x \varphi(s)^{-1}$ and thus
$\varphi(s) = t(\xi)z$ for some $z\in C_G(x)$. Set $\bar{s} = \varphi(s) \in X$.

\begin{lem}\label{l:key}
There exists a $C_G(x)$-conjugate of $X$ containing $x$ and $t(\xi)$.
\end{lem}

\begin{proof}
As noted above, we have $\bar{s} = t(\xi)z$ for some $z\in C_G(x)$. By Proposition \ref{l:exp1} there are scalars $\gamma_i \in K$ such that $z = \prod_{i=1}^{r}x_{i}(\gamma_i)$.
Let us consider a general element $y = \prod_{i=1}^{r}x_{i}(\delta_i) \in C_G(x)$. In view of (\ref{e:tac}), we get
\begin{eqnarray*}
y\bar{s}y^{-1} = \prod_{i}x_{i}(\delta_i)t(\xi)\prod_{i}x_{i}(\gamma_i - \delta_i)t(\xi)^{-1}t(\xi)  \!\!\! & = & \!\!\! \prod_{i}x_{i}(\delta_i)\prod_{i}x_{i}(\xi^{d_i}(\gamma_i - \delta_i))t(\xi)\\
\!\!\! & = & \!\!\! \prod_{i}x_{i}(\delta_i+\xi^{d_i}(\gamma_i-\delta_i))t(\xi)
\end{eqnarray*}
where the $d_i$ are the integers appearing in Table \ref{tab:di}.

Since $p \geqs h$ and $\mathbb{F}_{p}^{\times} = \la \xi \ra$, it is easy to see that there is at most one $i$ such that $\xi^{d_i}=1$. If there is no such $i$ then we can set $\delta_i = \xi^{d_i}\gamma_i/(\xi^{d_i}-1)$ for all $i$, so  
$y\bar{s}y^{-1} = t(\xi)$ and $X^y$ is the desired conjugate of $X$. Finally, suppose $\xi^{d_j}=1$ and $\gamma_j \ne 0$ for some $j$. By defining $\delta_i$ as above for all $i \ne j$, we get 
$y\bar{s}y^{-1} = t(\xi)x_{j}(\gamma_j)$ with $[t(\xi),x_{j}(\gamma_j)]=1$. But this implies that $y\bar{s}y^{-1}$ is a non-semisimple element, which  contradicts the  semisimplicity of $\bar{s}$.
 \end{proof}

In view of the lemma, we may assume that $X$ contains $t(\xi)$, which corresponds to a diagonalizable element $s \in {\rm SL}_{2}(p)$ with eigenvalues $\xi$ and $\xi^{-1}$. Since $t(\xi) \in T$, we can use the known action of $A$ on $V$ (see the proof of Proposition \ref{p:11}) to determine the eigenvectors and eigenspaces of $s$ on $V$. For example, 
\begin{equation}\label{e:cvx}
\{\xi^{d_1},\xi^{d_2}, \ldots, \xi^{d_r}\}
\end{equation}
is the collection of eigenvalues of $s$ on $C_V(x)$, where the $d_i$ are given in Table \ref{tab:di}. We set $\bar{s} = t(\xi) = sZ \in X$, where $Z$ is the centre of ${\rm SL}_{2}(p)$. Note that $A$ contains the Borel subgroup $\la \bar{s},x\ra$ of $X$. 

The proof of Theorem \ref{t:main} has three main steps, which we now describe.

\vs

\noindent \emph{Step 1: Elimination.} Our initial aim is to reduce to the situation where the action of $X$ on $V$ is compatible with the decomposition of $V$ as an $A$-module given in Table \ref{t:dec}. In almost all cases, we are able to achieve this goal. To do this, we consider the possible decompositions of $V|_{X}$ as a direct sum of indecomposable $KX$-modules, using the description of these modules given in Section \ref{s:rep}, with the aim of eliminating all but one possibility. 

First we use the fact that the decomposition of $V|_{X}$ has to be compatible with the Jordan form of $x$ on $V$ (this can be read off from the relevant tables in \cite{Lawther}). In addition, it  must be compatible with the known eigenvalues of $s$ on $V$ (as noted above, these are just the eigenvalues of $t(\xi)$ on $V$, which we can compute from the known action of $A$ on $V$). Note that if $M$ is an indecomposable summand of $V|_{X}$ then the restriction of $M$ to $\la s \ra$ is completely reducible, so we just need to identify the $KX$-composition factors of $M$ in order to compute the eigenvalues of $s$ on this summand. Often it is sufficient to compare the eigenvalues of $s$ on $C_V(x)$ with the expected eigenvalues in (\ref{e:cvx}), and we can also use our earlier calculations on the traces of elements of order $2$ and $3$ to obtain further restrictions on $V|_{X}$ (see Section \ref{ss:trace}). With this approach in mind, the following lemma will be useful.

\begin{lem}\label{l:eiggg}
Let $M$ be an indecomposable $KX$-module of the form $L_X(i)$, $U$ or $W(j)$, where $i \in \{0,2, \ldots, p-1\}$ and $j \in \{2,4,\ldots, p-3\}$. Then the eigenvalues of $s$ on $C_M(x)$ are $\xi^i$, $1$ and $\{\xi^j, \xi^{-j}\}$, respectively. 
\end{lem}

\begin{proof}
First recall that $x$ has Jordan form $[J_{i+1}]$, $[J_p]$ and $[J_p^2]$ on $L_X(i)$, $U$ and $W(j)$, respectively. The fixed point of $x$ on the simple module $L_X(i)$ has highest weight $i$, so the result is clear in this case. Similarly, ${\rm soc}(U) = L_X(0)$ so $s$ has eigenvalue $\xi^0$ on $C_U(x)$. Finally, suppose $M = W(j)$. The highest weight of ${\rm soc}(M) = L_X(j)$ is $j$, so $\xi^j$ is one of the eigenvalues of $s$ on $C_M(x)$. To determine the second eigenvalue, it is helpful to view $W(j)$ as the restriction to $X$ of the tilting module $T(2p-2-j)$ for the ambient algebraic group of type $A_1$ (see \cite[Lemma 2.3]{Seitz}). On the latter module, $x$ has a fixed point of weight $2p-2-j$ (the high weight), so the eigenvalue of $s$ is $\xi^{2p-2-j} = \xi^{-j}$ as required.
 \end{proof} 

Let us illustrate how Step 1 is carried out in the specific case $(G,p) = (E_8,31)$.

\begin{ex}\label{e:e8}
Suppose $G=E_8$ and $p=31$, so $x$ has Jordan form $[J_{31}^8]$ on $V$ (see \cite[Table 9]{Lawther}). In particular, $V|_X$ is projective and thus every indecomposable summand of $V|_X$ is also projective. In terms of the notation introduced in Section \ref{s:rep}, the possibilities for $V|_{X}$ are as follows 
$$\left\{\begin{array}{l}
M_1 \oplus M_2 \oplus M_3 \oplus M_4 \oplus M_5 \oplus M_6 \oplus M_7 \oplus M_8 \\
W(a_1) \oplus M_1 \oplus M_2 \oplus M_3 \oplus M_4 \oplus M_5 \oplus M_6 \\
W(a_1) \oplus W(a_2) \oplus M_1 \oplus M_2 \oplus M_3 \oplus M_4 \\
W(a_1) \oplus W(a_2) \oplus W(a_3) \oplus M_1 \oplus M_2 \\
W(a_1) \oplus W(a_2) \oplus W(a_3) \oplus W(a_4)
\end{array}\right.$$
where $M_{i} \in \{L_X(30), U\}$ and $a_i \in \{2,4,\ldots, 28\}$. If $V|_{X}$ has an $M_i$ summand, then $s$ has an eigenvalue $\xi^{30} = \xi^{0}$ on $C_V(x)$, which contradicts (\ref{e:cvx}), so we must have   
$$V|_{X} = W(a_1) \oplus W(a_2) \oplus W(a_3) \oplus W(a_4).$$ 
Since $\xi^{30}=1$ and $s$ has eigenvalues $\xi^{i}, \xi^{-i}$ on $C_{W(i)}(x)$ (see Lemma \ref{l:eiggg}), it follows that 
$$\{\xi^{a_1}, \xi^{-a_1}, \xi^{a_2}, \xi^{-a_2}, \xi^{a_3}, \xi^{-a_3}, \xi^{a_4}, \xi^{-a_4}\} = \{\xi^{28},\xi^{16},\xi^{8},\xi^{4},\xi^{26},\xi^{22},\xi^{14},\xi^2\}$$
(see Table \ref{tab:di}).
Up to a re-ordering of summands, this immediately implies that 
$$a_1 \in \{2,28\},\; a_2 \in \{4,26\},\; a_3 \in \{8,22\},\; a_4 \in \{14,16\}.$$ 
Let $y \in X$ be an involution. Since ${\rm tr}(W(i),y) =\pm 2$ and ${\rm tr}(V,y) \in \{-8,24\}$ (see Lemma \ref{l:3} and Proposition \ref{p:trace}), it follows that ${\rm tr}(W(a_i),y) = -2$ for all $i$, whence $a_i \equiv 2 \imod{4}$ and thus $(a_1,a_2,a_3,a_4) = (2,26,22,14)$ is the only possibility. We have now reduced to the case where the decomposition of $V|_{X}$ is compatible with $V|_{A}$ (see Table \ref{t:dec}).
\end{ex}

\vs

\noindent \emph{Step 2: Extension.} Next observe that if $V|_{X}$ has the decomposition given in Table \ref{t:dec} then the socle of $V|_{X}$ has a simple summand $W=L_X(2)$. To complete the argument, we aim to show that $W$ is an $\mathfrak{sl}_2$-subalgebra of $V$ and its stabilizer in $G$ is an $A_1$-type subgroup. We can do this in almost every case; the exceptions are the two special cases appearing in the statement of Theorem \ref{t:main}. 

Let $\{w_2, w_0, w_{-2}\}$ be a basis for $W$, where $w_i$ is an eigenvector for $s$ with eigenvalue $\xi^i$. We may assume that the action of $x$ on $W$ is given by the matrix
\begin{equation}\label{e:xmat}
\left(\begin{array}{ccc}
1 & 1 & 1 \\
0 & 1 & 2 \\
0 & 0 & 1 
\end{array}\right)
\end{equation}
with respect to this basis (that is, $x(w_0) = w_0+w_2$, etc.). If we define $\bar{s} = sZ \in X$ as above then $\la \bar{s},x\ra$ is a Borel subgroup of $X$ and we can consider the opposite Borel subgroup $\la \bar{s},x' \ra$ of $X$, where $x' \in X$ is also a regular unipotent element of order $p$. With respect to the above basis, we may assume that $x'$ acts on $W$ via the matrix
\begin{equation}\label{e:xxmat}
\left(\begin{array}{ccc}
1 & 0 & 0 \\
2 & 1 & 0 \\
1 & 1 & 1 
\end{array}\right)
\end{equation}
If all these conditions are satisfied, then we will say that $\{w_2, w_0, w_{-2}\}$ is a \emph{standard basis} for $W$.

With the aid of {\sc Magma} \cite{magma} we can construct a $\dim G \times \dim G$ matrix to represent the action of $x$ on $V$ with respect to our Chevalley basis $\mathcal{B}$. Let us illustrate this with an example.

\begin{ex}\label{ex:g2}
For $(G,p)=(G_2,7)$ we proceed as follows in {\sc Magma}:

\vspace{1mm}

\begin{verbatim}
G:=GroupOfLieType("G2",Rationals());             
L:=LieAlgebra(G);
e,f,h:=ChevalleyBasis(L);
I1:=[1..6]; I2:=[1..2];

B:=[f[7-i] : i in I1] cat [e[i]*f[i] : i in I2] cat [e[i] : i in I1];
L:=ChangeBasis(L,B);
B:=Basis(L);
e:=[B[8+i] : i in I1]; f:=[B[7-i] : i in I1]; h:=[B[6+i]: i in I2];

ad:=AdjointRepresentation(L);
y:=ad(e[1]+e[2]);
A:=MatrixAlgebra(Rationals(),14);
x:=Identity(A);  y:=A!y; 
for i in [1..10] do x:=x+(1/Factorial(i))*y^i; end for; 
B:=MatrixAlgebra(GF(7),14);
x:=B!x;
\end{verbatim}

In this example, we are working with a Chevalley basis 
\[
\mathcal{B}=\{e[i] , f[i], h[j] \,:\, i \in \{1, \ldots, 6\}, j \in \{1,2\}\}
\]
where $e[i]$ spans the root space of the $i$-th positive root, $f[i]$ is in the root space of the corresponding negative root, and $h[j] = [e[j],f[j]]$ for $j = 1,2$, with respect to the following ordering 
\[
\a_1, \; \a_2, \; \a_1+\a_2, \; 2\a_1+\a_2, \; 3\a_1+\a_2, \; 3\a_1+2\a_2
\]
of positive roots (note that this agrees with the ordering given by the {\sc Magma} command {\small \textsf{PositiveRoots(G)}}). We adopt an analogous set-up in all cases.
\end{ex}

Moreover, we can use Proposition \ref{p:a1key}(iv) to compute the eigenvalues and eigenvectors of $t(\xi)$ (and thus $s$) on $V$ in terms of $\mathcal{B}$. For $i \in \mathbb{Z}$, it will be convenient to write $E_i$ for the $\xi^i$-eigenspace of $s$ on $V$ (so $w_i \in E_i$ for the elements in a standard basis of $W$). 

Next we identify a basis $\{v_1, \ldots, v_r\}$ of the $1$-eigenspace $C_V(x) = \ker(x-1)$ in terms of $\mathcal{B}$, where $v_i \in E_{d_i}$ (see Table \ref{tab:di}). Since $w_2 \in C_V(x) \cap E_2$ we can write
$$w_2 = \sum_{i=1}^{r}a_i v_i$$
for some $a_i \in K$, where $a_i \ne 0$ only if $\xi^{d_i} = \xi^2$. Similarly, 
\begin{eqnarray*}
w_0  \!\!\! & \in & \!\!\! \left(\ker((x-1)^2) \setminus \ker(x-1)\right) \cap E_0, \\
w_{-2} \!\!\! & \in & \!\!\! \left(\ker((x-1)^3) \setminus \ker((x-1)^2)\right) \cap E_{-2}.
\end{eqnarray*}
Using {\sc Magma} it is straightforward to compute bases for the relevant kernels; these computations can be done by hand, but it is much quicker and more efficient to use a machine. 

Given these bases, say $\mathcal{B}_2$, $\mathcal{B}_0$ and $\mathcal{B}_{-2}$, we can write
$$w_2 = \sum_{v \in \mathcal{B}_2}a_vv,\;\;   w_0 = \sum_{v \in \mathcal{B}_0}b_vv,\;\;   w_{-2} = \sum_{v \in \mathcal{B}_{-2}}c_vv$$
for $a_v,b_v,c_v \in K$ and our goal is to determine these scalars. To do this, we can use the specified actions of $x$ and $x'$ on $W$ to derive relations between the coefficients. Further relations can be determined by exploiting the fact that $x$ and $x'$ are regular unipotent elements. For example, we observe that $x' \cdot w_{-2} = w_{-2}$ and 
$$x'\cdot [w_{-2},w_0] = [w_{-2},w_0+w_{-2}] = [w_{-2},w_0],$$
where $[\, ,\,]$ is the Lie bracket on $V$, 
so $w_{-2}, [w_{-2},w_0] \in C_V(x')$. Since the regularity of $x'$ implies that $C_G(x')$ is abelian, it follows that $C_V(x') = {\rm Lie}(C_G(x'))$ is an abelian subalgebra of $V$ (for the latter equality, recall that $p \geqs h$) and thus   
\begin{equation}\label{e:wm}
[w_{-2}, [w_{-2},w_0]] = 0.
\end{equation}

Proceeding in this way, our goal is to reduce to the case where $W = \la w_2,w_0,w_{-2}\ra$ is an $\mathfrak{sl}_2$-subalgebra, with $w_2 = \sum_{i=1}^{r}e_i$ and $w_{-2} \in \sum_{i=1}^{r}\Z f_i$. Moreover, we want to find integers $\l,\mu$ such that $(w_2,\l w_0,\mu w_{-2})$ is an $\mathfrak{sl}_2$-triple over $\Z$ (that is, an $\mathfrak{sl}_2$-triple of $\mathcal{L}_{\Z}$ in the notation of Section \ref{ss:a1}). Indeed, if we can do this, then Proposition \ref{p:new1} implies that the stabilizer of $W$ in $G$ is an $A_1$-type subgroup and so we are in the generic situation described in part (i) of Theorem \ref{t:main}. In a few cases, we are unable to force $w_{-2} \in \sum_{i=1}^{r}\Z f_i$, but by appealing to Proposition \ref{p:new2} we can still show that the same conclusion holds. 

In the remaining cases where $W$ is not an $\mathfrak{sl}_2$-subalgebra, or the action of $X$ on $V$ is incompatible with $V|_{A}$, we will show that $X$ stabilizes a non-zero subalgebra of $\la e_{\a}\mid \a \in \Phi^+(G)\ra$. More precisely, we will establish the following result, which reduces the proof of Theorem \ref{t:main} to the handful of cases appearing in Table \ref{tab:red} (see Remark \ref{r:main}(a) for the conjugacy statement in part (i)). 

\begin{thm}[Reduction Theorem]\label{t:red}
Let $G$ be a simple exceptional algebraic group of adjoint type over an algebraically closed field of characteristic $p >0$. Let $X = {\rm PSL}_{2}(p)$ be a subgroup of $G$ containing a regular unipotent element of $G$ and set $V = {\rm Lie}(G)$ with Chevalley basis as in (\ref{e:chev}). Then one of the following holds:
\begin{itemize}\addtolength{\itemsep}{0.2\baselineskip}
\item[{\rm (i)}] $X$ is contained in an $A_1$-type subgroup of $G$ and $X$ is uniquely determined up to $G$-conjugacy;
\item[{\rm (ii)}] $X$ stabilizes a non-zero subalgebra of $\la e_{\a} \mid \a \in \Phi^{+}(G)\ra$ and $(G,p,V|_{X})$ is one of the cases in Table $\ref{tab:red}$.
\end{itemize}
\end{thm}

\begin{table}
$$\begin{array}{lll} \hline
G & p  & V|_{X} \\ \hline
F_4 & 13  & W(10) \oplus W(2) \\
E_6 & 13  & W(10) \oplus W(8) \oplus W(2) \\
&  & W(10) \oplus W(4) \oplus W(2) \\
& &  W(10)^2 \oplus W(4) \\
E_7 & 19  & W(8) \oplus W(4) \oplus W(2) \oplus U\\
& &  W(16) \oplus W(10) \oplus W(4) \oplus U \\
& &  W(16) \oplus W(14) \oplus W(8) \oplus U \\
E_8 & 37  & W(34) \oplus W(26) \oplus W(14) \oplus L_X(22) \oplus L_X(2) \\ \hline
\end{array}$$
\caption{The exceptional cases $(G,p,V|_{X})$ in Theorem \ref{t:red}}
\label{tab:red}
\end{table}

We will prove the Reduction Theorem in Sections \ref{s:g2}--\ref{s:e8}, considering each possibility for $G$ in turn. 

\vs

\noindent \emph{Step 3: Parabolic analysis.} The final step in our proof of Theorem \ref{t:main} concerns the cases arising in Theorem \ref{t:red}(ii), given in Table \ref{tab:red}. In view of Proposition \ref{l:bt}, we may assume that $X$ is contained in a proper parabolic subgroup $P=QL$ of $G$ and we proceed by studying the possible embeddings of $X$ in such a subgroup. Take $P$ to be a minimal such parabolic and let $\pi:P \to P/Q$ be the quotient map. By identifying $L$ with $P/Q$, we may view $\pi(X)$ as a subgroup of $L'$. Now we can show that $\pi(X)<H$, where $H$ is an $A_1$-type subgroup of $L'$ containing a regular unipotent element of $L'$ (namely, $\pi(x)$), so we can use \cite[Tables 1--5]{LT99} to study the composition factors of $V|_{H}$ for each $(G,L')$. In turn, this imposes restrictions on the decomposition of $V|_{X}$. But the possibilities for $V|_{X}$ are listed in Table \ref{tab:red} and in this way we arrive at the two special cases in the statement of Theorem \ref{t:main}. See Section \ref{s:final} for the details. (Notice that we adopt a similar approach in the proof of Theorem \ref{t:main0} below.)

\begin{ex}\label{e:e9}
To illustrate some of the above ideas, let us explain how we handle the case $(G,p)=(E_8,31)$. Recall that in Example \ref{e:e8} we reduced to the situation where 
$$V|_X = W(2)\oplus W(26)\oplus W(22) \oplus W(14),$$
which is compatible with the decomposition of $V|_{A}$. By following the approach in Example \ref{ex:g2}, we use {\sc Magma} to determine the action of $x$ on $V$ in terms of a Chevalley basis $\mathcal{B}$.  

Let $W = {\rm soc}(W(2)) = L_X(2)$ and let $\{w_2, w_0, w_{-2}\}$ be a standard basis of $W$ as above. First consider $w_2 \in C_V(x)$. Now $C_V(x)\cap E_2$ is $1$-dimensional (indeed, by inspecting Table \ref{tab:di} we see that there is a unique $d_i$ which is congruent to $2$ modulo $30$), spanned by the sum of the simple root vectors, so we must have
$$w_2 = a_1(e_1+e_2+e_3+e_4+e_5+e_6+e_7+e_8)$$
for some non-zero scalar $a_1 \in K$. Similarly, $w_0$ is contained in the $1$-dimensional space $\ker((x-1)^2) \cap E_0$ and by considering the relation $x(w_0) = w_0+w_2$ we take
$$w_0 = a_2(h_1 + 19h_2+4h_3+9h_4+28h_5+18h_6+10h_7+4h_8).$$

Finally, $w_{-2}$ is in the $2$-dimensional space $\ker((x-1)^3) \cap E_{-2}$ (note that $\xi^{-2} = \xi^{58}$ since $p=31$) and it follows that  
$$w_{-2} = a_3(8f_1+28f_2+f_3+10f_4+7f_5+20f_6+18f_7+f_8) + a_4e_{\a_0}$$
where $\a_0  = 2\a_1+3\a_2+4\a_3+6\a_4+5\a_5+4\a_6+3\a_7+2\a_8$ is the highest root. Note that $a_3 \ne 0$ since $w_{-2} \in \ker((x-1)^3) \setminus \ker((x-1)^2)$ and $e_{\a_0} \in C_V(x)$. 

By considering the action of $x$ on $W$ (see (\ref{e:xmat})) we quickly deduce that $a_2 = 16a_1$ and $a_3 = 4a_1$. Finally, one checks that the condition in (\ref{e:wm}) yields $a_4=0$, so setting $a_1=1$ we have
\begin{eqnarray*}
w_2 \!\!\! & = & \!\!\! e_1+e_2+e_3+e_4+e_5+e_6+e_7+e_8 \\
w_0 \!\!\! & = & \!\!\! 16(h_1 + 19h_2+4h_3+9h_4+28h_5+18h_6+10h_7+4h_8) \\
w_{-2} \!\!\! & = & \!\!\! 4(8f_1+28f_2+f_3+10f_4+7f_5+20f_6+18f_7+f_8)
\end{eqnarray*}
and it is easy to see that $w_2, w_0$ and $w_{-2}$ satisfy the relations
\begin{equation}\label{e:br}
[w_2,w_{-2}] = 2w_0,\;\; [w_2,w_0] = w_2,\;\; [w_0,w_{-2}] = w_{-2},
\end{equation}
and thus $W = \la w_2, w_0, w_{-2}\ra$ is an $\mathfrak{sl}_{2}$-subalgebra. If we set 
\begin{equation}\label{e:br0}
w_2'=w_2, \;\; w_0' = -2w_0,\;\; w_{-2}' = -w_{-2},
\end{equation}
then $(w_2', w_0', w_{-2}')$ is an $\mathfrak{sl}_{2}$-triple. Moreover, working mod $p$, we have
$$w_{-2}' = 92f_1+136f_2+182f_3+270f_4+220f_5+168f_6+114f_7+58f_8$$
and thus $(w_2', w_0', w_{-2}')$ is an $\mathfrak{sl}_{2}$-triple over $\Z$ (see the proof of \cite[Proposition 2.4]{T}). Since $X$ stabilizes $W$, it is contained in an $A_1$-type subgroup of $G$ by Proposition \ref{p:new1}. This completes the proof of Theorem \ref{t:main} for $G=E_8$ with $p=31$. 
\end{ex}

We close this section by presenting a proof of Theorem \ref{t:main0}. 

\vs

\noindent \emph{Proof of Theorem \ref{t:main0}.}
Let $V = {\rm Lie}(G)$ be the adjoint module for $G$. Seeking a contradiction, suppose $X < P$, where $P=QL$ is a proper parabolic subgroup of $G$ with unipotent radical $Q$ and Levi factor $L$. We may as well assume that $P$ is minimal with respect to the containment of $X$. In particular, if $\pi:P \to P/Q$ is the quotient map and we identify $L$ with $P/Q$, then $\pi(X)$ is not contained in a proper parabolic subgroup of $L'$.  Now $\pi(x)$ is a regular unipotent element of $L'$ (see \cite[Lemma 2.6]{TZ}). Writing $L'=L_1 \cdots L_t$, where each $L_i$ is a simple factor, let $\pi_i:L'\to L_i$ be the naturally defined projection map. Then $\pi_i(\pi(X)) < L_i$ contains a regular unipotent element of $L_i$ and does not lie in a proper parabolic subgroup of $L_i$.

If $L_i$ is of classical type, we apply the main theorem of \cite{ST1} to see that $\pi_i(\pi(X))$ is contained in an $A_1$-type subgroup of $L_i$. On the other hand, if $L_i$ is of exceptional type, then $G$ is of type $E_n$ and $L_i$ is of type $E_m$ for $m<n$. In this case, we apply Theorem \ref{t:red} to conclude once again that $\pi_i(\pi(X))$ is contained in an $A_1$-type subgroup of $L_i$ for all relevant values of $p$. In particular, in all cases we deduce that $\pi(X)$ lies in an $A_1$-type subgroup $H$ of $L'$.

Now the $KH$-composition factors of $V|_{H}$ can be read off from the information in \cite[Tables 1--5]{LT99} and we can use this to determine the $KX$-composition factors of $V|_{X}$ (to do this, note that we may set all $q_i=1$ in terms of the notation in \cite[Tables 1--5]{LT99}). Indeed, each composition factor of $V|_{P}$ is an irreducible $KL'$-module (the unipotent radical $Q$ acts trivially on the $KP$-composition factors of $V|_{P}$), so the decompositions of $V|_{X}$ and $V|_{H}$ have to be compatible. But the decomposition of $V|_{X}$ is given in Table \ref{t:dec} and in this way we will reach a contradiction.

To see this, first observe that $X$ has at least one trivial composition factor on $V$, coming from $Z(L)$. By inspecting Table \ref{t:dec}, this immediately implies that
$$(G,p) \in \{(F_4,13), (E_6,13), (E_8,37)\}.$$
Suppose $(G,p) = (E_8,37)$. From Table \ref{t:dec}, the $KX$-composition factors of $V|_{X}$ are as follows:
\begin{equation}\label{e:e8dee}
L_X(34)^2, L_X(26)^2, L_X(22)^2, L_X(20), L_X(14)^2, L_X(10), L_X(8), L_X(2)^2, L_X(0).
\end{equation}
By inspecting \cite[Table 5]{LT99}, using the fact that $V|_{X}$ has a unique trivial composition factor, we deduce that $L' = A_4A_2A_1$, $A_4A_3$ or $D_5A_2$. However, in each of these cases we see that $V|_{X}$ has an $L_X(6)$ composition factor, which is incompatible with (\ref{e:e8dee}). The other two possibilities for $(G,p)$ can be eliminated in a similar fashion. For example, if $(G,p) = (E_6,13)$ then the composition factors of $V|_{X}$ are 
$$L_X(10)^3, L_X(8)^3, L_X(4), L_X(2)^4, L_X(0).$$ 
By inspecting \cite[Table 3]{LT99}, just considering trivial composition factors, we deduce that $L' =A_2^2A_1$, $A_4A_1$ or $D_5$, but in each case we find that $V|_{X}$ has two or more $L_X(4)$ factors. This is a contradiction.

\vs

As mentioned above, the proof of Theorem \ref{t:red} will be given in Sections \ref{s:g2}--\ref{s:e8}, where we carry out Steps 1 and 2 (elimination and extension) for each group in turn. We handle Step 3 in Section \ref{s:final}, thus completing the proof of Theorem \ref{t:main}.

\section{The case $G=G_2$}\label{s:g2}

We begin the proof of Theorem \ref{t:main} by handling the case $G=G_2$. As noted in Remark \ref{r:main}(c), the result in this case can be deduced from the proof of \cite[Lemma 3.1]{ST2} (it also follows from Kleidman's classification of the maximal subgroups of $G_2(p)$ in \cite{K}).

\begin{thm}\label{t:g2}
Let $G$ be a simple algebraic group of type $G_2$ over an algebraically closed field of characteristic $p>0$. Let $X = {\rm PSL}_{2}(p)$ be a subgroup of $G$ containing a regular unipotent element $x$ of $G$. Then $X$ is contained in an $A_1$-type subgroup of $G$.
\end{thm}

\begin{proof}
The Coxeter number of $G$ is $7$, so we have $p \geqs 7$. Let $V = {\rm Lie}(G)$ be the adjoint module for $G$ and fix a Chevalley basis for $V$ as in (\ref{e:chev}). We will use  the notation introduced in Section \ref{ss:me}. In particular, $\la \bar{s}, x\ra$ is a Borel subgroup of $X$, where $\bar{s} = t(\xi) = sZ$ and 
\begin{equation}\label{e:eigg2}
\{\xi^{10},\xi^{2}\}
\end{equation} 
are the eigenvalues of $s \in {\rm SL}_{2}(p)$ on $C_V(x)$, where $\mathbb{F}_{p}^{\times} = \la \xi \ra$. Let $E_i$ be the $\xi^i$-eigenspace of $s$ on $V$ and recall from Section \ref{ss:a1} that we may assume $x$ is obtained by exponentiating the regular nilpotent element $e = e_{1} + e_{2} \in V$ (that is, we will assume $x = {\rm exp}({\rm ad}(e))$). According to \cite[Table 2]{Lawther}, the Jordan form of $x$ on $V$ is as follows:
\begin{equation}\label{e:200}
\left\{\begin{array}{ll}
\mbox{$[J_{11},J_{3}]$} & p \geqs 11 \\
\mbox{$[J_{7}^2]$} & p=7. 
\end{array}\right.
\end{equation}
We will use the notation $U$ and $W(i)$ for the projective indecomposable $KX$-modules defined in (\ref{e:u}) and (\ref{e:wi}), respectively. 

\vs

\noindent \emph{Case 1. $V|_{X}$ is semisimple} 

\vs

First assume $V|_{X}$ is semisimple and recall that $x$ has Jordan form $[J_{m+1}]$ on $L_X(m)$ (for $0 \leqs m <p$). In view of (\ref{e:200}), it follows that 
$$V|_{X} = \left\{\begin{array}{ll}
L_X(10) \oplus L_X(2) & p \geqs 11 \\
L_X(6)^2 & p = 7.
\end{array}\right.$$
If $p=7$ then the above decomposition implies that $\xi^6$ is an eigenvalue of $s$ on $C_V(x)$, but this is not compatible with (\ref{e:eigg2}).

Now assume $p \geqs 11$, so 
$V|_{X} = L_X(10) \oplus L_X(2)$. As in Section \ref{ss:me}, let 
$\{w_2,w_0,w_{-2}\}$ be a standard basis for the summand $W = L_X(2)$, so $w_i \in E_i$ and the action of $x$ on $W$ is given by the matrix in (\ref{e:xmat}). Our goal is to show that $W$ is an $\mathfrak{sl}_2$-subalgebra with $w_2=e$ and $w_{-2} \in \sum_{i=1}^{2}\Z f_i$. Furthermore, we seek integers $\l,\mu$ so that $(w_2,\l w_0, \mu w_{-2})$ is an $\mathfrak{sl}_2$-triple over $\Z$, which will allow us to apply Proposition \ref{p:new1}. 

For $p \geqs 17$ we find that each space
\begin{equation}\label{e:xe}
\ker(x-1) \cap E_2,\;\; \ker((x-1)^2) \cap E_0,\;\; \ker((x-1)^3) \cap E_{-2}
\end{equation}
is $1$-dimensional, which gives us  
$$w_2 = a_1(e_1 + e_2),\; w_0 = a_2(h_1 + 13h_2),\; w_{-2} = a_3(4f_1+f_2)$$
for some non-zero scalars $a_i \in K$ (in the expressions for $w_0$ and $w_{-2}$, the specific coefficients of the $h_i$ and $f_i$ will depend on the characteristic $p$; the coefficients presented here are for $p=17$). If we set $a_1=1$ then by considering the action of $x$ on $V$ we deduce that $a_2=14$ and $a_3 = 7$. Now $w_2, w_0$ and $w_{-2}$ satisfy the relations in (\ref{e:br}), so we get an $\mathfrak{sl}_2$-triple $(w_2',w_0',w_{-2}')$ 
as in (\ref{e:br0}). Here $w_{-2}'=-w_{-2} = 6f_1+10f_2$ (for $p=17$) and thus $(w_2',w_0',w_{-2}')$ is an $\mathfrak{sl}_2$-triple over $\Z$ (see the proof of \cite[Proposition 2.4]{T}). Finally, by applying Proposition \ref{p:new1}, we conclude that $X$ is contained in an $A_1$-type subgroup of $G$.

Next assume $p=13$. Once again $\ker(x-1) \cap E_2$ and $\ker((x-1)^2) \cap E_0$ are $1$-dimensional, but now $\ker((x-1)^3) \cap E_{-2}$ is $2$-dimensional, spanned by the vectors $11f_1+f_2$ and $e_{32}$ (here we use the notation $e_{32}$ for $e_{\gamma}$ with $\gamma = 3\alpha_1+2\alpha_2$). Therefore
$$w_2 = a_1(e_1 + e_2),\; 
w_0  = a_2(h_1 + 6h_2),\;
w_{-2}  = a_3(11f_1+f_2)+a_4e_{32}
$$
for some $a_i \in K$. By considering the action of $x$ on $W$ we deduce that $a_2 = 10a_1$ and $a_3 = 3a_1$. Moreover, (\ref{e:wm}) implies that  $a_4=0$ and by arguing as above, setting $a_1=1$ and using Proposition \ref{p:new1}, we deduce that $X$ is contained in an $A_1$-type subgroup of $G$.

Now suppose $p=11$. Here we have
$$
w_2  = a_1(e_1 + e_2),\; 
w_0  = a_2(h_1 + 9h_2)+a_3e_{32}, \; 
w_{-2} = a_4(5f_1+f_2)+a_5e_{31}
$$
and by considering the action of $x$ on $W$ we deduce that $a_2 = 8a_1$, $a_4 = a_1$ and $a_5 = 2a_3$. We may as well set $a_1 = 1$, so 
$$
w_2  = e_1 + e_2,\; 
w_0  = 8(h_1 + 9h_2)+\gamma e_{32},\;
w_{-2}  = 5f_1+f_2+2\gamma e_{31}
$$
for some $\gamma \in K$. One now checks that the relations in (\ref{e:br}) are satisfied (for all $\gamma$), so $W = \la w_2,w_0,w_{-2}\ra$ is an $\mathfrak{sl}_{2}$-subalgebra. Moreover, if we take 
$$e=w_2, \; h=-2(8(h_1+9h_2)) = 6h_1+4h_2,\; f = -(5f_1+f_2) = 6f_1+10f_2,$$ 
then $(e,h,f)$ is an $\mathfrak{sl}_2$-triple over $\Z$ and we can apply Proposition \ref{p:new2} (with $y = e_{32}$ and $z=e_{31}$). It follows that  the stabilizer of $W$ in $G$ is an $A_1$-type subgroup.

\vs

\noindent \emph{Case 2. ${\rm rad}(V|_{X}) \ne 0$, $p \geqs 11$} 

\vs

To complete the proof of the theorem we may assume that ${\rm rad}(V|_{X}) \neq 0$. Suppose $p \geqs 11$ and $W$ is a reducible indecomposable summand of $V|_{X}$. If $p \geqs 13$ then $\dim W \geqs 12$ (see Corollary \ref{c:bound}) and thus Lemma \ref{l:jf} implies that $x$ has a Jordan block of size $n \geqs 12$ on $W$, but this is incompatible with (\ref{e:200}). Now assume $p=11$. Here (\ref{e:200}) implies that $W$ has at least three composition factors (if there were only two, then Lemma \ref{l:jf} and Corollary \ref{c:2step} would imply that $x$ has Jordan form $[J_{10}]$ or $[J_{11},J_{1}]$ on $W$, which contradicts  (\ref{e:200})). By Lemma \ref{l:jf}, it follows that $x$ has Jordan form $[J_{11},J_{i}]$ on $W$ with $i \in \{0,3\}$, so $\dim W \in \{11,14\}$. By considering Theorem \ref{t:structure}, it is easy to see that $i=0$ is the only possibility, so $W = U$ is projective and thus 
$$V|_{X} = U \oplus L_X(2).$$
However, this implies that an involution $x_2 \in X$ has trace $2$ on $V$ (see Section \ref{ss:trace}), which is incompatible with Proposition \ref{p:trace}. This is a contradiction.

\vs

\noindent \emph{Case 3. ${\rm rad}(V|_{X}) \ne 0$, $p=7$} 

\vs

Finally, let us assume $p=7$. Let $P = \langle x \rangle$ be a Sylow $p$-subgroup of $X$ and observe that $V|_{P}$ is projective. Then \cite[Corollary 3, Section 9]{Alperin} implies that $V|_{X}$ is projective and thus each indecomposable summand is also projective. Since the eigenvalues of $s$ on $C_V(x)$ are $\{\xi^{4}, \xi^2\}$, we deduce that $V|_{X} = W(2)$ or $W(4)$. In fact, by considering the trace of $x_2$, we see that $V|_{X} = W(2)$ is the only option. This is compatible with the decomposition of $V$ with respect to an $A_1$-type subgroup of $G$ containing a regular unipotent element (see Table \ref{t:dec}).

Let $W$ be the $L_X(2)$ summand in the socle of $V|_{X}$ and let $\{w_2,w_0,w_{-2}\}$ be a standard basis. The spaces $\ker(x-1) \cap E_2$ and $\ker((x-1)^2) \cap E_0$ are $1$-dimensional, whereas $\ker((x-1)^3) \cap E_{-2}$ is $2$-dimensional and we get  
$$
w_2  = a_1(e_{1} + e_{2}), \; 
w_0  = a_2(h_1+4h_2), \; 
w_{-2} = a_3(2f_1+f_2) + a_4e_{32}
$$
for some $a_i \in K$. Set $a_1=1$, so $w_2=e$. By considering the action of $x$ on $W$ we deduce that $a_2 = a_3 = 4$. Moreover, (\ref{e:wm}) implies that $a_4=0$ and we deduce that $W = \la w_2,w_0,w_{-2}\ra$ is an $\mathfrak{sl}_2$-subalgebra and the relations in (\ref{e:br}) are satisfied. As before, the desired result now follows by applying Proposition \ref{p:new1}. 

\vs

This completes the proof of Theorem \ref{t:g2}.
 \end{proof}

\section{A reduction for $G=F_4$}\label{s:f4}

In this section our goal is to establish Theorem \ref{t:red} when $G=F_4$. The proof of Theorem \ref{t:main} in this case will be completed in Section \ref{s:final}. Our main result is the following.

\begin{thm}\label{t:f4}
Let $G$ be a simple algebraic group of type $F_4$ over an algebraically closed field of characteristic $p >0$. Let $X = {\rm PSL}_{2}(p)$ be a subgroup of $G$ containing a regular unipotent element $x$ of $G$ and set $V = {\rm Lie}(G)$. Then one of the following holds:
\begin{itemize}\addtolength{\itemsep}{0.2\baselineskip}
\item[{\rm (i)}] $X$ is contained in an $A_1$-type subgroup of $G$;
\item[{\rm (ii)}] $p=13$, $V|_{X} = W(10) \oplus W(2)$ and $X$ stabilizes a non-zero subalgebra of $\la e_{\a} \mid \a \in \Phi^{+}(G)\ra$.
\end{itemize}
\end{thm}

\begin{proof}
Here $p \geqs 13$ and we set up the standard notation as before. In particular, 
\begin{equation}\label{e:eigf4}
\{\xi^{22},\xi^{14},\xi^{10},\xi^{2}\}
\end{equation} 
are the eigenvalues of $s$ on $C_V(x)$, where $\mathbb{F}_{p}^{\times} = \la \xi \ra$, and  
\begin{equation}\label{e:2}
\left\{\begin{array}{ll}
\mbox{$[J_{23},J_{15},J_{11},J_{3}]$} & p \geqs 23 \\
\mbox{$[J_{19}^2,J_{11},J_{3}]$} & p=19 \\
\mbox{$[J_{17}^2,J_{15},J_{3}]$} & p=17 \\
\mbox{$[J_{13}^4]$} & p=13 
\end{array}\right.
\end{equation}
is the Jordan form of $x$ on $V$ (see \cite[Table 4]{Lawther}). We may assume that $x$ is obtained by exponentiating the regular nilpotent element $e = e_{1} + e_{2} +e_{3} +e_{4}$ in $V$, with respect to a Chevalley basis for $V$ as in (\ref{e:chev}). It will also be useful to note that $V|_{X}$ is self-dual. 

\vs

\noindent \emph{Case 1. $V|_{X}$ is semisimple} 

\vs

If $p \in \{13,17,19\}$ then (\ref{e:2}) implies that  
$$V|_{X} = \left\{\begin{array}{ll}
L_X(18)^2 \oplus L_X(10) \oplus L_X(2) & p = 19 \\
L_X(16)^2 \oplus L_X(14) \oplus L_X(2) & p = 17 \\
L_X(12)^4 & p = 13 
\end{array}\right.$$
but none of these decompositions are compatible with the eigenvalues of $s$ on $C_V(x)$ given in (\ref{e:eigf4}). For example, if $p=19$ then the given decomposition implies that the relevant eigenvalues are $\{\xi^0,\xi^0, \xi^{10},\xi^2\}$, but this contradicts (\ref{e:eigf4}).

Now assume $p \geqs 23$, so 
$$V|_{X} = L_X(22) \oplus L_X(14) \oplus L_X(10) \oplus L_X(2).$$
Let $W$ be the $L_X(2)$ summand and let $\{w_2,w_0,w_{-2}\}$ be a standard basis for $W$ as in Section \ref{ss:me}, so $w_i \in E_i$ (the $\xi^i$-eigenspace of $s$ on $V$) and the action of $x$ and $x'$ on $W$ is given by the matrices in (\ref{e:xmat}) and (\ref{e:xxmat}), respectively, where $\la \bar{s},x'\ra$ is the opposite Borel subgroup of $X$. If $p \geqs 29$ then the spaces in (\ref{e:xe}) are $1$-dimensional and we get
\begin{eqnarray*}
w_2  \!\!\! & = & \!\!\! a_1(e_1 + e_2+e_3+e_4) \\
w_0  \!\!\! & = & \!\!\! a_2(h_1 + 23h_2+4h_3+6h_4) \\
w_{-2} \!\!\! & = & \!\!\! a_3(5f_1+28f_2+20f_3+f_4) 
\end{eqnarray*}
for some $a_i \in K$ (in the expressions for $w_0$ and $w_{-2}$, the specific coefficients depend on the characteristic $p$; the ones given here are for $p=29$). If we set $a_1=1$ then we can use the action of $x$ on $V$ to deduce that $a_2=18$ and $a_3=13$. Moreover, the relations in (\ref{e:br}) are satisfied and it follows that $(w_2',w_0',w_{-2}')$ is an $\mathfrak{sl}_2$-triple, where these elements are defined in (\ref{e:br0}). Now
$$w_{-2}' = -w_{-2} = -13(5f_1+28f_2+20f_3+f_4) = 22f_1+42f_2+30f_3+16f_4$$
working mod $p$ (for $p=29$), so $(w_2',w_0',w_{-2}')$ is an $\mathfrak{sl}_2$-triple over $\Z$ (see the proof of \cite[Proposition 2.4]{T}). By applying Proposition \ref{p:new1}, we conclude that $X$ is contained in an $A_1$-type subgroup of $G$.

Now suppose $p=23$. Here
\begin{eqnarray*}
w_2 \!\!\! & = & \!\!\! a_1(e_1+e_2+e_3+e_4) \\
w_0 \!\!\! & = & \!\!\! a_2(h_1 + 4h_2+16h_3+7h_4)+a_3e_{2342} \\
w_{-2} \!\!\! & = & \!\!\! a_4(10f_1+17f_2+22f_3+f_4) +a_5e_{1342} 
\end{eqnarray*}
for some $a_i \in K$ (we use the notation $e_{2342}$ for 
$e_{\gamma}$ with $\gamma = 2\a_1+3\a_2+4\a_3+2\a_4$, and similarly for $e_{1342}$). By considering the action of $x$ on $W$ we deduce that $a_2=12a_1$, $a_4 = 7a_1$ and $a_5=2a_3$. 
Setting $a_1=1$ we get
\begin{eqnarray*}
w_2 \!\!\! & = & \!\!\! e_1+e_2+e_3+e_4 \\
w_0 \!\!\! & = & \!\!\! 12(h_1 + 4h_2+16h_3+7h_4)+\gamma e_{2342} \\
w_{-2} \!\!\! & = & \!\!\! 7(10f_1+17f_2+22f_3+f_4) +2\gamma e_{1342} 
\end{eqnarray*}
for some $\gamma \in K$, and one can check that the relations in (\ref{e:br}) are satisfied. In particular, $W$ is an $\mathfrak{sl}_2$-subalgebra of $V$. Set 
$$e=w_2, \; h = -2(12(h_1 + 4h_2+16h_3+7h_4)) = 22h_1+15h_2+14h_3+9h_4$$
and
$$f = -(7(10f_1+17f_2+22f_3+f_4)) = 22f_1+19f_2+7f_3+16f_4.$$
Then $(e,h,f)$ is an $\mathfrak{sl}_2$-triple over $\Z$ and by applying Proposition \ref{p:new2} (with $y = e_{2342}$ and $z = e_{1342}$) we deduce that the stabilizer of $W$ in $G$ is an $A_1$-type subgroup.

\vs

\noindent \emph{Case 2. ${\rm rad}(V|_{X}) \neq 0$, $p \geqs 19$} 

\vs

For the remainder we may assume that ${\rm rad}(V|_{X}) \neq 0$. First assume $p \geqs 19$. By arguing as in Case 2 in the proof of Theorem \ref{t:g2}, it is straightforward to reduce to the case $p=19$. For example, suppose $p=23$ and $W$ is a reducible indecomposable summand of $V|_{X}$. The Jordan form of $x$ on $V$ (see (\ref{e:2})) implies that $W$ has at least three composition factors and we can use Lemma \ref{l:jf} to see that $x$ has Jordan form $[J_{23},J_{i}]$ on $W$ with $i \in \{0,3,11,15\}$, so $\dim W \in \{23,26,34,38\}$. Using Theorem \ref{t:structure}, we deduce that $i=0$ is the only option, so 
$$V|_{X} = U \oplus L_X(14) \oplus L_X(10) \oplus L_X(2).$$
But this implies that an involution $x_2 \in X$ has trace $0$ on $V$, which is incompatible with Proposition \ref{p:trace}. 

Now assume $p=19$. Suppose $W$ is a reducible non-projective indecomposable summand of $V|_{X}$. By combining Lemma \ref{l:jf} and Theorem \ref{t:structure} we deduce that $x$ has Jordan form $[J_{19}^2, J_{11}]$ or $[J_{19}^2, J_{3}]$ on $W$, so there is a unique such summand (and the other summand is simple). However, this is incompatible with the self-duality of $V|_{X}$. For example, if $x$ has Jordan form $[J_{19}^2,J_3]$ on $W$, then $V|_{X} = W \oplus L_X(10)$ and Theorem \ref{t:structure} implies that 
$${\rm soc}(W) = L_X(0) \oplus L_X(2) \oplus L_X(4),\;\; W/{\rm soc}(W) = L_X(16) \oplus L_X(14)$$
(up to duality) so $V|_{X}$ is not self-dual. 

Therefore, we may assume that each indecomposable summand is either simple or projective, so the possibilities for $V|_{X}$ are as follows:
$$\left\{\begin{array}{l}
U \oplus L_X(18) \oplus L_X(10) \oplus L_X(2) \\
U^2 \oplus L_X(10) \oplus L_X(2) \\
W(i) \oplus L_X(10) \oplus L_X(2)
\end{array}\right.$$
with $i \in \{2,4,\ldots, 16\}$. As in (\ref{e:eigf4}), the eigenvalues of $s$ on $C_V(x)$ are 
\[
\{\xi^4, \xi^{14},\xi^{10},\xi^2\}.
\] 
Since $s$ has eigenvalues $1$ and $\{\xi^{i},\xi^{-i}\}$ on $C_U(x)$ and $C_{W(i)}(x)$, respectively (see Lemma \ref{l:eiggg}), it follows that 
$$V|_{X} = W(i)\oplus L_X(10) \oplus L_X(2)$$ 
with $i \in \{4,14\}$. The case $i=4$ can be ruled out by considering the trace of $x_2$; hence $i=14$ and $V|_{X}$ is compatible with the containment of $X$ in an $A_1$-type subgroup of $G$ (see Table \ref{t:dec}). We need to show that $X$ is contained in such a subgroup. To do this we can repeat the argument in Case 1 for $p \geqs 29$ (the details are entirely similar).

\vs

\noindent \emph{Case 3. ${\rm rad}(V|_{X}) \neq 0$, $p=17$} 

\vs

Now assume $p=17$. Suppose $W$ is a reducible non-projective indecomposable summand of $V|_{X}$. It is easy to check that the Jordan form of $x$ on $W$ is either $[J_{17}^2,J_3]$ or $[J_{17},J_{15}]$, so there is a unique such summand. If $x$ has Jordan form $[J_{17}^2,J_3]$ on $W$ then Theorem \ref{t:structure} implies that $V|_{X} = W \oplus L_X(14)$ (up to duality) where 
$${\rm soc}(W) = L_X(0) \oplus L_X(2) \oplus L_X(4),\;\; W/{\rm soc}(W) = L_X(16) \oplus L_X(14),$$
but this is incompatible with the self-duality of $V|_{X}$. Similarly, in the other case we have $V|_{X} = W \oplus V_1 \oplus L_X(2)$ and 
$${\rm soc}(W) = L_X(i) \oplus L_X(i+2),\;\; W/{\rm soc}(W) = L_X(14-i) \oplus L_X(12-i)$$
with $i \in \{0,2,4, \ldots, 12\}$ and $V_1 \in \{ L_X(16), U \}$. By self-duality, $i=6$ is the only option. But this implies that $x_2$ has trace $0$ on $V$, which contradicts Proposition \ref{p:trace}.

It follows that each indecomposable summand of $V|_{X}$ is either simple or projective. By arguing as above (the case $p=19$), using the fact that $s$ has eigenvalues $\{\xi^{6},\xi^{14},\xi^{10},\xi^2\}$ on $C_V(x)$, we deduce that 
$$V|_{X} = W(i) \oplus L_X(14) \oplus L_X(2)$$ 
with $i \in \{6,10\}$. If $i=6$ then one can check that an element $x_3 \in X$ of order $3$ has trace $1$ on $V$, so Proposition \ref{p:trace} implies that $i=10$. Therefore, the action of $X$ is compatible with an $A_1$-type subgroup of $G$ (see Table \ref{t:dec}) and it remains to establish the desired containment. 

As before, let $\{w_2, w_0, w_{-2}\}$ be a standard basis of the $L_X(2)$ summand $W$ in the decomposition of $V|_{X}$. In the usual manner we deduce that 
\begin{eqnarray*}
w_2 \!\!\! & = & \!\!\! a_1(e_1 + e_2+e_3+e_4) \\
w_0 \!\!\! & = & \!\!\! a_2(h_1 + 5h_2+6h_3+10h_4) 
\end{eqnarray*}
for some non-zero scalars $a_1,a_2 \in K$. We may assume $a_1=1$. Now $w_{-2}$ is contained in $\ker((x-1)^3) \cap E_{-2}$, which is $2$-dimensional, and we get 
$$w_{-2} = a_3(12f_{1}+9f_{2}+4f_{3}+f_{4}) + a_4(e_{1231} - e_{1222}).$$
Since the action of $x$ on $W$ is given by the matrix in (\ref{e:xmat}) we deduce that $a_2=6$ and $a_3=1$. Finally, the condition in (\ref{e:wm}), which is obtained  by considering the action of $x'$ on $W$, implies that $a_4=0$. It is easy to see that the relations in (\ref{e:br}) are satisfied and we complete the argument in the usual manner, via Proposition \ref{p:new1}.

\vs

\noindent \emph{Case 4. ${\rm rad}(V|_{X}) \neq 0$, $p=13$} 

\vs

Finally, let us assume that $p=13$. Here $V|_{X}$ is projective and thus each indecomposable summand is also projective. Since the eigenvalues of $s$ on $C_V(x)$ are $\{\xi^{10}, \xi^2, \xi^{10}, \xi^2\}$, we quickly deduce that $V|_{X}$ is one of the following:
$$W(2) \oplus W(10), \; W(2)^2, \; W(10)^2.$$ 
Let $y  = \hat{y}Z \in X$ be an element of order $7$, where $\hat{y} \in {\rm SL}_{2}(13)$ is ${\rm SL}_{2}(K)$-conjugate to a diagonal matrix ${\rm diag}(\omega,\omega^{-1})$ and $\omega \in K$ is a non-trivial $7$-th root of unity. For each decomposition we can compute the eigenvalues of $y$ on $V$ and then compare the results with the list of eigenvalue multiplicities of all elements in $G$ of order $7$ (as noted in Remark \ref{r:AJL}, the latter can be computed using Litterick's algorithm in \cite{AJL}). For example, if $V|_{X} = W(2)^2$ then $y \in {\rm GL}_{52}(K)$ is conjugate to the diagonal matrix 
$$[I_8, \omega I_8, \omega^2I_8,\omega^3I_6,\omega^4I_6,\omega^5I_8,\omega^6I_8],$$ 
but one checks that no element in $G$ of order $7$ acts on $V$ with these eigenvalues. In this way, we deduce that $V|_{X} = W(2)\oplus W(10)$ is the only possibility. 
 
Let $W$ be the $L_X(2)$ summand in the socle of $V|_{X}$ and let $\{w_2,w_0,w_{-2}\}$ be a standard basis. The spaces $\ker(x-1) \cap E_2$ and $\ker((x-1)^2) \cap E_0$ are $2$-dimensional and we get
\begin{eqnarray*}
w_2 \!\!\! & = & \!\!\! a_1(e_{1} + e_{2} +e_{3} +e_{4}) + a_2(e_{1231}-e_{1222}) \\
w_0 \!\!\! & = & \!\!\! a_3(h_1+9h_2+12h_3+9h_4) + a_4(e_{1221}+10e_{1122})
\end{eqnarray*}
for some $a_i \in K$. Finally, one checks that $\ker((x-1)^3) \cap E_{-2}$ is $4$-dimensional and we take 
$$w_{-2} = a_5(3f_{1}+f_{2}+10f_{3}+f_{4})+a_6(e_{1220}+3e_{0122}) +a_7(e_{1121}+8e_{0122})+a_8e_{2342}.$$
In the usual manner, by considering the action of $x$ on $W$, we get $a_3 = 2a_1$, $a_4=a_2$, $a_5 = 10a_1$ and $a_7 = 2a_2+a_6$. In addition, the condition in (\ref{e:wm}) yields  the following system of equations:
\begin{eqnarray*}
a_1^2a_2 + 8a_1^2a_6 \!\!\!\! & = & \!\!\! 0 \\
7a_1^2a_2 + 11a_1^2a_6 \!\!\!\! & = & \!\!\! 0 \\
11a_1^2a_2 + 5a_1^2a_6 \!\!\!\! & = & \!\!\! 0 \\ 
4a_1^2a_8 + 8a_1a_2^2 + 7a_1a_2a_6 \!\!\!\! & = & \!\!\! 0.
\end{eqnarray*}
If $a_1 \ne 0$ then these equations imply that $a_2=a_6=a_8=0$, so 
\begin{eqnarray*}
w_2 \!\!\! & = & \!\!\! a_1(e_{1} + e_{2} +e_{3} +e_{4}) \\
w_0 \!\!\! & = & \!\!\! -a_1(h_1+9h_2+12h_3+9h_4) \\
w_{-2} \!\!\! & = & \!\!\! 10a_1(3f_{1}+f_{2}+10f_{3}+f_{4})
\end{eqnarray*}
and by setting $a_1=1$ we can use Proposition \ref{p:new1} to show that $X$ is contained in an $A_1$-type subgroup. On the other hand, if $a_1=0$ then we can set $a_2=1$, so 
\begin{eqnarray*}
w_2 \!\!\! & = & \!\!\! e_{1231}-e_{1222} \\
w_0 \!\!\! & = & \!\!\! e_{1221}+10e_{1122} \\
w_{-2} \!\!\! & = & \!\!\! a_6(e_{1220}+3e_{0122}) +(2+a_6)(e_{1121}+8e_{0122})+a_8e_{2342}.
\end{eqnarray*}
One checks that $[w_2,w_0]=0$, so 
$$0 = x' \cdot [w_2,w_0] = [w_2+2w_0+w_{-2}, w_0+w_{-2}]$$
and we deduce that $a_6=7$, hence
$$w_{-2} = 7e_{1220}+9e_{1121}+2e_{0122}+a_8e_{2342}.$$
It is now easy to check that $W \subseteq \la e_{\a} \mid \a \in \Phi^+(G)\ra$ is a subalgebra, which gives case (ii) in the statement of the theorem.

\vs

This completes the proof of Theorem \ref{t:f4}.
 \end{proof}

\section{A reduction for $G=E_6$}\label{s:e6}

The following result, which we prove in this section, establishes Theorem \ref{t:red} for groups of type $E_6$.

\begin{thm}\label{t:e6}
Let $G$ be a simple adjoint algebraic group of type $E_6$ over an algebraically closed field of characteristic $p >0$. Let $X = {\rm PSL}_{2}(p)$ be a subgroup of $G$ containing a regular unipotent element $x$ of $G$ and set $V = {\rm Lie}(G)$. Then one of the following holds:
\begin{itemize}\addtolength{\itemsep}{0.2\baselineskip}
\item[{\rm (i)}] $X$ is contained in an $A_1$-type subgroup of $G$; 
\item[{\rm (ii)}] $p=13$, $V|_{X}$ is one of 
$$W(10) \oplus W(8) \oplus W(2), \; W(10) \oplus W(4) \oplus W(2), \; W(10)^2 \oplus W(4)$$
and $X$ stabilizes a non-zero subalgebra of $\la e_{\a} \mid \a \in \Phi^{+}(G)\ra$.
\end{itemize}
\end{thm}

\begin{proof}
Here $p \geqs 13$, $V|_{X}$ is self-dual and 
\begin{equation}\label{e:eige6}
\{\xi^{22},\xi^{16},\xi^{14},\xi^{10},\xi^{8},\xi^2\}
\end{equation} 
are the eigenvalues of $s$ on $C_V(x)$, where  $\mathbb{F}_{p}^{\times} = \la \xi \ra$ (see Section \ref{ss:me}). By inspecting \cite[Table 6]{Lawther}, we see that 
\begin{equation}\label{e:4}
\left\{\begin{array}{ll}
\mbox{$[J_{23},J_{17},J_{15},J_{11},J_9,J_3]$} & p \geqs 23 \\
\mbox{$[J_{19}^2,J_{17},J_{11},J_9,J_3]$} & p=19 \\
\mbox{$[J_{17}^3,J_{15},J_9,J_{3}]$} & p=17 \\
\mbox{$[J_{13}^6]$} & p=13 
\end{array}\right.
\end{equation}
is the Jordan form of $x$ on $V$. We adopt the notation introduced in Section \ref{s:prel}.

\vs

\noindent \emph{Case 1. $V|_{X}$ is semisimple} 

\vs

If $p \in \{13,17,19\}$ then 
$$V|_{X} = \left\{\begin{array}{ll}
L_X(18)^2 \oplus L_X(16) \oplus L_X(10) \oplus L_X(8) \oplus L_X(2) & p = 19 \\
L_X(16)^3 \oplus L_X(14) \oplus L_X(8) \oplus L_X(2) & p =17  \\
L_X(12)^6 & p = 13
\end{array}\right.$$
but not one of these decompositions is compatible with the eigenvalues of $s$ on $C_V(x)$ (see (\ref{e:eige6})) so we may assume $p \geqs 23$ and
$$V_X = L_X(22) \oplus L_X(16) \oplus L_X(14) \oplus L_X(10) \oplus L_X(8) \oplus L_X(2).$$
Let $W$ be the $L_X(2)$ summand and let $\{w_2,w_0,w_{-2}\}$ be a standard basis for $W$. If $p \geqs 29$ then one checks that each of the spaces in (\ref{e:xe}) are $1$-dimensional and the result quickly follows via 
Proposition \ref{p:new1}. For example, if $p=29$ then  
\begin{eqnarray*}
w_2 \!\!\! & = & \!\!\! a_1(e_1 + e_2+e_3+e_4+e_5+e_6) \\
w_0 \!\!\! & = & \!\!\! a_2(h_1+5h_2+20h_3+28h_4+20h_5+h_6) \\
w_{-2} \!\!\! & = & \!\!\! a_3(f_1+5f_2+20f_3+28f_4+20f_5+f_6)
\end{eqnarray*}
and by setting $a_1=1$ and considering the action of $x$ on $W$ (see (\ref{e:xmat})), we deduce that $a_2=21$ and $a_3=13$. One now checks that $(w_{2},-2w_0,-w_{-2})$ is an $\mathfrak{sl}_2$-triple over $\Z$ (see the proof of \cite[Proposition 2.4]{T}) and by applying Proposition \ref{p:new1} we deduce that $X$ is contained in an $A_1$-type subgroup of $G$. 

Now assume $p=23$. Here $\ker((x-1)^2) \cap E_0$ and $\ker((x-1)^3) \cap E_{-2}$ are $2$-dimensional and we get
\begin{eqnarray*}
w_2 \!\!\! & = & \!\!\! a_1(e_1 + e_2+e_3+e_4+e_5+e_6) \\
w_0 \!\!\! & = & \!\!\! a_2(h_1 + 10h_2+22h_3+17h_4+22h_5+h_6)+a_3e_{122321} \\
w_{-2} \!\!\! & = & \!\!\! a_4(f_1+10f_2+22f_3+17f_4+22f_5+f_6) +a_5e_{112321} 
\end{eqnarray*}
where $a_1a_2a_4 \ne 0$. Set $a_1=1$. From the action of $x$ on $W$ we deduce that $a_2=15$, $a_4 = 7$ and $a_5=2a_3$, so 
\begin{eqnarray*}
w_2 \!\!\! & = & \!\!\! e_1 + e_2+e_3+e_4+e_5+e_6 \\
w_0 \!\!\! & = & \!\!\! 15(h_1 + 10h_2+22h_3+17h_4+22h_5+h_6)+\gamma e_{122321} \\
w_{-2} \!\!\! & = & \!\!\! 7(f_1+10f_2+22f_3+17f_4+22f_5+f_6) +2\gamma e_{112321} 
\end{eqnarray*}
for some $\gamma \in K$. If we take
$$e = w_2, \; h = -2(15(h_1 + 10h_2+22h_3+17h_4+22h_5+h_6))$$
and
\begin{eqnarray*}
f \!\!\! & = & \!\!\! -7(f_1+10f_2+22f_3+17f_4+22f_5+f_6) \\
& = & \!\!\! 16f_1+22f_2+7f_3+19f_4+7f_5+16f_6
\end{eqnarray*}
then $(e,h,f)$ is an $\mathfrak{sl}_2$-triple over $\Z$ and using Proposition \ref{p:new2} we conclude that $X$ is contained in an $A_1$-type subgroup of $G$.

\vs

\noindent \emph{Case 2. ${\rm rad}(V|_{X}) \neq 0$, $p \geqs 19$} 

\vs

If $p \geqs 23$ then we can essentially repeat the argument in the proof of Theorem \ref{t:f4} (see the first paragraph in Case 2). Indeed, it is easy to reduce to the case where $p=23$ and 
$$V|_{X} = U \oplus L_X(16) \oplus L_X(14) \oplus L_X(10) \oplus L_X(8) \oplus L_X(2),$$
but this is not compatible with (\ref{e:eige6}). 

Now assume $p=19$.  Suppose $V|_{X}$ has a reducible non-projective indecomposable summand $W$. By applying Lemma \ref{l:jf} and Theorem \ref{t:structure}, we deduce that the Jordan form of $x$ on $W$ is one of the following:
$$[J_{19}^{2},J_{11}], \; [J_{19}^{2},J_{9}],  \; [J_{19}^{2},J_{3}], \;[J_{19},J_{17}].$$ 
In particular, $V|_{X}$ has a unique such summand. The structure of $W$ is described in Theorem \ref{t:structure} and it is easy to see that the existence of such a summand contradicts the self-duality of $V|_{X}$. For instance, suppose $x$ has Jordan form $[J_{19}^2,J_9]$ on $W$. Then up to duality we have 
$${\rm soc}(W) = L_X(6) \oplus L_X(8) \oplus L_X(10),\;\; W/{\rm soc}(W) = L_X(10) \oplus L_X(8)$$
and thus $V|_{X} = W \oplus L_X(16) \oplus L_X(10) \oplus L_X(2)$ is not self-dual. The other cases are very similar.

Therefore, we may assume that each indecomposable summand of $V|_{X}$ is either simple or projective. By considering the eigenvalues of $s$ in (\ref{e:eige6}), we deduce that 
$$V|_{X} = W(i) \oplus L_X(16) \oplus L_X(10) \oplus L_X(8) \oplus L_X(2)$$
with $i \in \{4,14\}$. If $i=4$ then we find that $x_2$ has trace $2$ on $V$, which contradicts Proposition \ref{p:trace}, hence $i=14$ is the only possibility. In the usual manner, we now  construct a basis 
\begin{eqnarray*}
w_2 \!\!\! & = & \!\!\! a_1(e_1+e_2+e_3+e_4+e_5+e_6) \\
w_0 \!\!\! & = & \!\!\! a_2(h_1+18h_2+9h_3+5h_4+9h_5+h_6) \\
w_{-2} \!\!\! & = & \!\!\! a_3(f_1+18f_2+9f_3+5f_4+9f_5+f_6) + a_4(e_{112211} - e_{111221})
\end{eqnarray*}
(with $a_1a_2a_3 \ne 0$) of the summand $W = L_X(2)$ of $V|_{X}$. If we set $a_1=1$ and consider the action of $x$ on $W$ (see (\ref{e:xmat})) we deduce that $a_2=11$ and $a_3 = 3$, and one checks that the condition in (\ref{e:wm}) gives $a_4=0$. The result now follows in the usual manner via Proposition \ref{p:new1}.

\vs

\noindent \emph{Case 4. ${\rm rad}(V|_{X}) \neq 0$, $p=17$} 

\vs

First assume that $V|_{X}$ has a reducible indecomposable non-projective summand $W$. In the usual way, by combining Lemma \ref{l:jf} and Theorem \ref{t:structure}, we deduce that the Jordan form of $x$ on $W$ is one of the following:
$$[J_{17}^{3},J_3], \; [J_{17}^{2},J_{9}], \; [J_{17}^{2},J_3], \; [J_{17},J_{15}].$$

Suppose that $x$ has Jordan form $[J_{17}^{3},J_3]$ on $W$, so $V|_{X} = W \oplus L_X(14) \oplus L_X(8)$. By applying Theorem \ref{t:structure}, using the self-duality of $V|_{X}$, we deduce that 
$${\rm soc}(W) = W/{\rm soc}(W) = L_X(10) \oplus L_X(8) \oplus L_X(6)$$
is the only possibility, but this is incompatible with the eigenvalues of $s$ on $C_V(x)$. We can rule out $[J_{17}^{2},J_{9}]$ and $[J_{17}^{2},J_3]$ by the self-duality of $V|_{X}$, so let us assume $x$ has Jordan form $[J_{17},J_{15}]$ on $W$. By self-duality it follows that
$${\rm soc}(W) = W/{\rm soc}(W) = L_X(8) \oplus L_X(6)$$
and thus $V|_{X}$ is one of the following:
$$\left\{\begin{array}{l}
W \oplus M_1 \oplus M_2 \oplus L_X(8) \oplus L_X(2) \\
W \oplus W(i) \oplus L_X(8) \oplus L_X(2) 
\end{array}\right.$$
where $M_j \in \{L_X(16),U\}$ and $i \in \{2,4, \ldots, 14\}$. However, it is clear that none of these decompositions are compatible with (\ref{e:eige6}). 

For the remainder, we may assume that each indecomposable summand of $V|_{X}$ is either simple or projective. By considering the eigenvalues of $s$, we deduce that 
$$V|_{X} = W(i) \oplus M_1 \oplus L_X(14) \oplus L_X(8) \oplus L_X(2)$$
with $i \in \{6,10\}$ and $M_1 \in \{L_X(16),U\}$. By computing the trace of $x_3$ and appealing to Proposition \ref{p:trace} (and also Remark \ref{r:eeig}), it follows that $i=10$ and $M_1=L_X(16)$ is the only possibility. In particular, we have now reduced to the case where the decomposition of $V|_{X}$ is compatible with containment in an $A_1$-type subgroup of $G$ (see Table \ref{t:dec}). 

As before, let $W$ be the $L_X(2)$ summand of $V|_{X}$ and let $\{w_2,w_0,w_{-2}\}$ be a standard basis of $W$. The reader can check that
\begin{eqnarray*}
w_2  \!\!\! & = & \!\!\! a_1(e_1+e_2+e_3+e_4+e_5+e_6) \\
w_0  \!\!\! & = & \!\!\! a_2(h_1+12h_2+4h_3+9h_4+4h_5+h_6) + a_3(e_{112211} - e_{111221}) \\
w_{-2} \!\!\! & =  & \!\!\! a_4(f_1+12f_2+4f_3+9f_4+4f_5+f_6) + a_5(e_{112210} + e_{011221}) \\
&  & \!\!\! + a_6(e_{111211} + 15e_{011221})
\end{eqnarray*}
with $a_1a_2a_4 \ne 0$. Set $a_1=1$. By considering the action of $x$ on $W$ we deduce that $a_2 = 9$, $a_4=1$ and $a_6 = 2a_3+a_5$. The condition in (\ref{e:wm}) yields $a_5 = 15a_3$, so $a_6=0$ and thus
\begin{eqnarray*}
w_2  \!\!\! & = & \!\!\! e_1+e_2+e_3+e_4+e_5+e_6 \\
w_0  \!\!\! & = & \!\!\! 9(h_1+12h_2+4h_3+9h_4+4h_5+h_6) + \gamma(e_{112211} - e_{111221}) \\
w_{-2} \!\!\! & = & \!\!\! f_1+12f_2+4f_3+9f_4+4f_5+f_6 + 15\gamma(e_{112210} + e_{011221})
\end{eqnarray*}
for some $\gamma \in K$. In addition, the relations in (\ref{e:br}) are satisfied and $W$ is an $\mathfrak{sl}_2$-subalgebra of $V$. Set
$$e = w_2, \; h = -2(9(h_1+12h_2+4h_3+9h_4+4h_5+h_6))$$
and 
\begin{eqnarray*}
f \!\!\! & = & \!\!\! -(f_1+12f_2+4f_3+9f_4+4f_5+f_6) \\
& = & \!\!\! 16f_1+5f_2+13f_3+8f_4+13f_5+16f_6.
\end{eqnarray*}
Then $(e,h,f)$ is an $\mathfrak{sl}_2$-triple over $\Z$ and by applying Proposition \ref{p:new2} (with $y = e_{112211} - e_{111221}$ and $z=-(e_{112210} + e_{011221})$) we conclude that $X$ is contained in an $A_1$-type subgroup of $G$.

\vs

\noindent \emph{Case 5. ${\rm rad}(V|_{X}) \neq 0$, $p=13$} 

\vs

Here (\ref{e:4}) implies that $V|_{X}$ is projective, so each indecomposable summand is also  projective. In view of (\ref{e:eige6}), we must have $V|_{X} = W(i) \oplus W(j) \oplus W(k)$ with $i,j \in \{2,10\}$ and $k \in \{4,8\}$. In each case, the traces of $x_2$ and $x_3$ are $-2$ and $-3$, respectively, so we need to work harder to eliminate some of these decompositions. Let $y = \hat{y}Z \in X$ be an element of order $7$, where $\hat{y} \in {\rm SL}_{2}(13)$ is ${\rm SL}_{2}(K)$-conjugate to a diagonal matrix ${\rm diag}(\omega,\omega^{-1})$ and $\omega \in K$ is a non-trivial $7$-th root of unity. We can compute the eigenvalues of $y$ on $V$ and then compare with the eigenvalue multiplicities of all elements in $G$ of order $7$, which we obtain using the algorithm in \cite{AJL}. In this way, we deduce that $V|_{X}$ is one of the following:
$$W(10)\oplus W(8) \oplus W(2),\; W(10)\oplus W(4) \oplus W(2),\; W(10)^2 \oplus W(4).$$

\vs

\noindent \emph{Case 5(a). $p=13$, $V|_{X} = W(10)\oplus W(8) \oplus W(2)$} 

\vs

Here $V|_{X}$ is compatible with the containment of $X$ in an $A_1$-type subgroup of $G$ (see Table \ref{t:dec}). Let $W$ be the $L_X(2)$ summand in the socle of $V|_{X}$ and let $\{w_2,w_0,w_{-2}\}$ be a standard basis. In the usual manner, we deduce that
\begin{eqnarray*}
w_2  \!\!\! & = & \!\!\! a_1(e_1+e_2+e_3+e_4+e_5+e_6) + a_2(e_{112210}+e_{111211}-e_{011221}) \\
w_0  \!\!\! & = & \!\!\! a_3(h_1+3h_2+10h_3+h_4+10h_5+h_6) + a_4(e_{111210}+3e_{111111}-e_{011211}) \\
w_{-2}  \!\!\! & = & \!\!\! a_5(f_1+3f_2+10f_3+f_4+10f_5+f_6) + a_6(e_{111110}+7e_{011210}-e_{011111}) \\
& & \!\!\! + a_7(e_{101111}+9e_{011210}) + a_8(e_{122321})
\end{eqnarray*}
for some scalars $a_i \in K$. By considering the action of $x$ on $W$, together with the condition in (\ref{e:wm}), we see that
$$a_3=5a_1, \;\; a_4 = a_2, \;\; a_5 = 10a_1,\;\; a_7 = 6a_2+2a_6$$
and either $a_1 = 0$ or $a_2=a_6=a_8=0$. In the latter situation, we set $a_1=1$ and then check that the relations in (\ref{e:br}) are satisfied -- this allows us to apply Proposition \ref{p:new1} to conclude that $X$ is contained in an $A_1$-type subgroup of $G$. Now assume $a_1=0$ and set $a_2=1$. Here one checks that $[w_{-2},w_2]=0$, so $[w_{-2},w_2+2w_0+w_{-2}]=0$ since $x'$ preserves the Lie bracket on $V$. This yields $a_6=9$, so
\begin{eqnarray}
w_2  \!\!\! & = & \!\!\! e_{112210}+e_{111211}-e_{011221}  \nonumber \\
w_0 \!\!\! & = & \!\!\! e_{111210}+3e_{111111}-e_{011211} \label{e:w0} \\
w_{-2} \!\!\! & = & \!\!\! 9e_{111110}+11e_{101111}+6e_{011210}+4e_{011111}+a_8(e_{122321}). \nonumber
\end{eqnarray}
We conclude that $W$ is an $X$-invariant subalgebra of $\la e_{\a} \mid \a \in \Phi^+(G) \ra$, as in part (ii) of the theorem.

\vs

\noindent \emph{Case 5(b). $p=13$, $V|_{X} = W(10)\oplus W(4) \oplus W(2)$ or $W(10)^2 \oplus W(4)$} 

\vs

Let $W$ be the $L_X(4)$ summand in the socle of $V|_{X}$ and let $\{w_4,w_2,w_0,w_{-2},w_{-4}\}$ be a basis of $W$ with $w_i \in E_i$. We may assume that the actions of $x$ and $x'$ on $W$ are given by the matrices
\begin{equation}\label{e:lx4}
\left(\begin{array}{ccccc}
1 & 1 & 1 & 1 & 1 \\
0 & 1 & 2 & 3 & 4 \\
0 & 0 & 1 & 3 & 6 \\
0 & 0 & 0 & 1 & 4 \\
0 & 0 & 0 & 0 & 1
\end{array}\right),\;\;\; 
\left(\begin{array}{ccccc}
1 & 0 & 0 & 0 & 0 \\
4 & 1 & 0 & 0 & 0 \\
6 & 3 & 1 & 0 & 0 \\
4 & 3 & 2 & 1 & 0 \\
1 & 1 & 1 & 1 & 1
\end{array}\right)
\end{equation}
respectively (in terms of this basis). One checks that $\ker(x-1) \cap E_4$ is $1$-dimensional, whereas the spaces
$$\ker((x-1)^2) \cap E_2,\; \ker((x-1)^3) \cap E_0$$
are $3$-dimensional, and  
$$\ker((x-1)^4) \cap E_{-2},\; \ker((x-1)^5) \cap E_{-4}$$
have dimension $5$ and $6$, respectively, and we get 
\begin{eqnarray*}
w_4 \!\!\! & = & \!\!\! a_1(e_{112211}-e_{111221}) \\
w_2 \!\!\! & = & \!\!\! a_2(e_{112210}+e_{011221}) + a_3(e_{111211}+11e_{011221}) \\
& & \!\!\! + a_4(e_1+e_2+e_3+e_4+e_5+e_6) \\
w_0 \!\!\! & = & \!\!\! a_5(e_{111210}+e_{011211}) + a_6(e_{111111}+8e_{011211}) \\ 
& & \!\!\! +a_7(h_1+3h_2+10h_3+h_4+10h_5+h_6) \\
w_{-2} \!\!\! & = & \!\!\! a_8(e_{111110}+e_{011111}) + a_9(e_{101111}+10e_{011111}) + a_{10}(e_{011210}+9e_{011111}) \\
& & \!\!\! + a_{11}(f_1+3f_2+10f_3+f_4+10f_5+f_6) + a_{12}(e_{122321}) \\
w_{-4} \!\!\! & = & \!\!\! a_{13}(e_{111100}+3e_{001111}) + a_{14}(e_{101110}+e_{001111}) + a_{15}(e_{011110}+6e_{001111}) \\
& & \!\!\! + a_{16}(e_{010111}+3e_{001111})+ a_{17}(12f_{101000}+10f_{010100}+2f_{001100} \\
& & \!\!\! +11f_{000110} +f_{000011}) + a_{18}(e_{112321})
\end{eqnarray*}
for some $a_i \in K$.

Set $a_1=1$ and consider the relations among the $a_i$ obtained from the action of $x$ on this basis. It is also helpful to note that $x'$ is a regular unipotent element, so $C_V(x')$ is abelian and we see that $[w_{-4},[w_{-4},w_{-2}]] = 0$ since $[w_{-4},w_{-2}]\in C_V(x')$. In this way, we deduce that 
\begin{eqnarray*}
w_4 \!\!\! & = & \!\!\! e_{112211}-e_{111221} \\
w_2 \!\!\! & = & \!\!\! a_2(e_{112210}+e_{011221}) + (1+a_2)(e_{111211}+11e_{011221}) \\
w_0  \!\!\! & = & \!\!\! 2a_2(e_{111210}+e_{011211}) + (6+6a_2)(e_{111111}+8e_{011211}) \\
w_{-2}  \!\!\! & = & \!\!\! a_8(e_{111110}+e_{011111}) + (4+5a_2+2a_8)(e_{101111}+10e_{011111}) \\
& & \!\!\! + (6a_2+12a_8)(e_{011210}+9e_{011111})+ a_{12}(e_{122321}) \\
w_{-4} \!\!\! & = & \!\!\! a_{13}(e_{111100}+3e_{001111}) + (2a_2+8a_8+a_{13})(e_{101110}+e_{001111}) \\
& & \!\!\! + (11a_2+9a_8)(e_{011110}+6e_{001111}) \\
& & \!\!\! + (1+2a_2+2a_8+12a_{13})(e_{010111}+3e_{001111}) + 4a_{12}(e_{112321}).
\end{eqnarray*}

Next one checks that $[w_2, w_{-2}]=0$, so $x'\cdot [w_2,w_{-2}]=0$ and thus 
$$[w_2+3w_0+3w_{-2}+w_{-4}, w_{-2}+w_{-4}] = 0$$
since $x'$ preserves the Lie bracket. This yields $a_{13} = 12a_2+2a_{2}^2+12a_{2}a_{8}$. Similarly, $[w_4, w_{2}]=0$ and thus 
$$[w_4+4w_2+6w_0+4w_{-2}+w_{-4}, w_{2}+3w_{0}+3w_{-2}+w_{-4}] = 0.$$
This relation implies that $a_2^2+12a_2a_8+9a_2+12a_8 = 0$ and it is now straightforward to check that $W \subseteq \la e_{\a} \mid \a \in \Phi^+(G)\ra$ is a subalgebra. 

\vs

This completes the proof of Theorem \ref{t:e6}.
 \end{proof}

\section{A reduction for $G=E_7$}\label{s:e7}

In this section we establish the following result, which proves Theorem \ref{t:red} for groups of type $E_7$.

\begin{thm}\label{t:e7}
Let $G$ be a simple adjoint algebraic group of type $E_7$ over an algebraically closed field of characteristic $p >0$. Let $X = {\rm PSL}_{2}(p)$ be a subgroup of $G$ containing a regular unipotent element $x$ of $G$ and set $V = {\rm Lie}(G)$. Then one of the following holds:
\begin{itemize}\addtolength{\itemsep}{0.2\baselineskip}
\item[{\rm (i)}] $X$ is contained in an $A_1$-type subgroup of $G$; 
\item[{\rm (ii)}] $p=19$, $V|_{X}$ is one of 
$$W(8) \oplus W(4) \oplus W(2) \oplus U, \;\; W(16) \oplus W(10) \oplus W(4) \oplus U,$$
$$W(16) \oplus W(14) \oplus W(8) \oplus U$$
and $X$ stabilizes a non-zero subalgebra of $\la e_{\a} \mid \a \in \Phi^{+}(G)\ra$.
\end{itemize}
\end{thm}

\begin{proof}
Here we have $p \geqs 19$ and  
\begin{equation}\label{e:eige7}
\{\xi^{34},\xi^{26},\xi^{22},\xi^{18},\xi^{14},\xi^{10},\xi^2\}
\end{equation} 
is the collection of eigenvalues of $s$ on $C_V(x)$, where $\mathbb{F}_{p}^{\times} = \la \xi \ra$. By \cite[Table 8]{Lawther}, the Jordan form of $x$ on $V$ is as follows:
\begin{equation}\label{e:5}
\left\{\begin{array}{ll}
\mbox{$[J_{35},J_{27},J_{23},J_{19},J_{15},J_{11},J_3]$} & p \geqs 37 \\
\mbox{$[J_{31}^2,J_{23},J_{19},J_{15},J_{11},J_3]$} & p = 31 \\
\mbox{$[J_{29}^2,J_{27},J_{19},J_{15},J_{11},J_{3}]$} & p=29 \\
\mbox{$[J_{23}^5,J_{15},J_{3}]$} & p=23 \\
\mbox{$[J_{19}^7]$} & p=19 
\end{array}\right.
\end{equation}
Note that $V|_{X}$ is self-dual.

\vs

\noindent \emph{Case 1. $V|_{X}$ is semisimple} 

\vs

If $p<37$ then the eigenvalues of $s$ on $C_V(x)$ are incompatible with (\ref{e:eige7}), so we may assume $p \geqs 37$ and thus 
$$V|_{X} = L_X(34) \oplus L_X(26) \oplus L_X(22) \oplus L_X(18) \oplus L_X(14) \oplus L_X(10) \oplus L_X(2)$$
in view of (\ref{e:5}). Let $W$ be the $L_X(2)$ summand and let $\{w_2,w_0,w_{-2}\}$ be a standard basis for $W$. In the usual manner, it is straightforward to show that $W$ is an appropriate $\mathfrak{sl}_2$-subalgebra and we can use Proposition \ref{p:new1} to show that (i) holds in the statement of the theorem. For example, if $p=37$ we get
\begin{eqnarray*}
w_2  \!\!\! & = & \!\!\!  a_1(e_1+e_2+e_3+e_4+e_5+e_6+e_7) \\
w_0 \!\!\! & = & \!\!\!  a_2(h_1+33h_2+15h_3+5h_4+12h_5+32h_6+28h_7)  \\
w_{-2} \!\!\! & = & \!\!\! a_3(4f_1+21f_2+23f_3+20f_4+11f_5+17f_6+f_7) + a_4e_{2 2 3 4 3 2 1}. 
\end{eqnarray*}
If we set $a_1=1$, then by considering the action of $x$ on $W$, we deduce that $a_2=20$ and $a_3=10$. Furthermore, the relation in (\ref{e:wm}) implies that $a_4=0$ and we deduce that $(w_2, -2w_0, -w_{-2})$ is an $\mathfrak{sl}_2$-triple over $\Z$ (see the proof of \cite[Proposition 2.4]{T}). Now apply Proposition \ref{p:new1}.

\vs

\noindent \emph{Case 2. ${\rm rad}(V|_{X}) \ne 0$, $p \geqs 29$} 

\vs

If $p \geqs 37$ then a combination of Lemma \ref{l:jf} and Corollary \ref{c:bound} implies that $x$ has a Jordan block of size $n \geqs 36$ on $V$, but this contradicts (\ref{e:5}). 

Next assume $p=31$. In the usual way, by applying Lemma \ref{l:jf} and Theorem \ref{t:structure}, and by appealing to the self-duality of $V|_{X}$, we can reduce to the case where each indecomposable summand of $V|_{X}$ is either simple or projective. By considering the eigenvalues in (\ref{e:eige7}), it follows that 
$$V|_{X} = W(i) \oplus L_X(22) \oplus L_X(18) \oplus L_X(14) \oplus L_X(10) \oplus L_X(2)$$
with $i \in \{4,26\}$. If $i=4$ then an involution $x_2 \in X$ has trace $-3$ on $V$, which contradicts Proposition \ref{p:trace}. Therefore $i=26$ and it is entirely straightforward to show that the $L_X(2)$ summand of $V|_{X}$ is an appropriate 
$\mathfrak{sl}_2$-subalgebra. The result follows via Proposition \ref{p:new1} in the usual fashion.

A similar argument applies when $p=29$. If $V|_{X}$ has a reducible non-projective summand then the self-duality of $V|_{X}$ implies that
$$V|_{X} = M_1 \oplus M_2 \oplus L_X(18) \oplus L_X(14) \oplus L_X(10) \oplus L_X(2)$$
is the only possibility, where 
$${\rm soc}(M_1) \cong M_1/{\rm soc}(M_1) = L_X(12) \oplus L_X(14)$$
and $M_2 \in \{L_X(28),U\}$. However, this implies that $x_2$ has trace $-3$ on $V$, which is a contradiction. Therefore, the indecomposable summands of $V|_{X}$ are simple or projective, and by considering the eigenvalues in (\ref{e:eige7}) we deduce that
$$V|_{X} = W(i) \oplus L_X(26) \oplus L_X(18) \oplus L_X(14) \oplus L_X(10) \oplus L_X(2)$$
with $i \in \{6,22\}$. We can rule out $i=6$ by computing the trace of $x_3$, so $i=22$ and we complete the argument as in the previous case. 

\vs

\noindent \emph{Case 3. ${\rm rad}(V|_{X}) \ne 0$, $p=23$} 

\vs

As before, it is not difficult to reduce to the case where each indecomposable summand of $V|_{X}$ is either simple or projective. By considering the eigenvalues in (\ref{e:eige7}) we deduce that 
$$V|_{X} = W(i) \oplus W(j) \oplus M_1 \oplus L_X(14) \oplus L_X(2)$$
where $i \in \{4,18\}$, $j \in \{10,12\}$ and $M_1 \in \{L_X(22), U\}$. By computing the trace of $x_2$ we see that $(i,j) = (4,12)$ or $(18,10)$, and we can rule out the first possibility by considering the trace of $x_3$. This calculation with $x_3$ also implies that $M_1 = L_X(22)$, so
$$V|_{X} = W(18) \oplus W(10) \oplus L_X(22) \oplus L_X(14) \oplus L_X(2).$$ 

Let $W$ be the $L_X(2)$ summand and fix a standard basis $\{w_2,w_0,w_{-2}\}$. By considering the spaces 
$$\ker(x-1) \cap E_2, \; \ker((x-1)^2) \cap E_0,\; \ker((x-1)^3) \cap E_{-2},$$
we deduce that
\begin{eqnarray*}
w_2  \!\!\! & = & \!\!\! a_1(e_1+e_2+e_3+e_4+e_5+e_6+e_7) \\
w_0 \!\!\! & = & \!\!\! a_2(h_1+17h_2+6h_3+15h_4+11h_5+11h_6+15h_7)  \\
& & \!\!\! + a_3(e_{1223210}+e_{1123211}-e_{1122221}) \\
w_{-2} \!\!\! & = & \!\!\! a_4(20f_1+18f_2+5f_3+f_4+13f_5+13f_6+f_7) \\
& & \!\!\! + a_5(e_{1123210}+2e_{1122211}+20e_{1112221}).
\end{eqnarray*}
Setting $a_1=1$ and using the action of $x$ on $W$, we deduce that $a_2 = 6$, $a_4=19$ and $a_5=2a_3$, so we have
\begin{eqnarray*}
w_2  \!\!\! & = & \!\!\! e_1+e_2+e_3+e_4+e_5+e_6+e_7 \\
w_0 \!\!\! & = & \!\!\! 6(h_1+17h_2+6h_3+15h_4+11h_5+11h_6+15h_7)  \\
& & \!\!\! + \gamma(e_{1223210}+e_{1123211}-e_{1122221}) \\
w_{-2} \!\!\! & = & \!\!\! 19(20f_1+18f_2+5f_3+f_4+13f_5+13f_6+f_7) \\
& & \!\!\! + 2\gamma(e_{1123210}+2e_{1122211}+20e_{1112221})
\end{eqnarray*}
for some $\gamma \in K$. One can check that the relations in (\ref{e:br}) are satisfied, so $W$ is an $\mathfrak{sl}_2$-subalgebra of $V$. Set
$$e = w_2, \; h = -2(6(h_1+17h_2+6h_3+15h_4+11h_5+11h_6+15h_7)),$$
\begin{eqnarray*}
f \!\!\! & = & \!\!\! -19(20f_1+18f_2+5f_3+f_4+13f_5+13f_6+f_7) \\
& = & \!\!\! 11f_1+3f_2+20f_3+4f_4+6f_5+6f_6+4f_7
\end{eqnarray*}
and
$$y = e_{1223210}+e_{1123211}-e_{1122221},\;\; z = e_{1123210}+2e_{1122211}+20e_{1112221}.$$
Then $(e,h,f)$ is an $\mathfrak{sl}_2$-triple over $\Z$ and we can use Proposition \ref{p:new2} to deduce that $X$ is contained in an $A_1$-type subgroup of $G$.

\vs

\noindent \emph{Case 4. ${\rm rad}(V|_{X}) \ne 0$, $p=19$} 

\vs

Finally, let us assume $p=19$ so $V|_{X}$ is projective and each indecomposable summand is also projective. By considering the eigenvalues in (\ref{e:eige7}), it follows that 
$$V|_{X} = W(i) \oplus W(j) \oplus W(k) \oplus M_1$$
where $i \in \{2,16\}$, $j \in \{4,14\}$, $k \in \{8,10\}$ and $M_1 \in \{L_X(18),U\}$. By computing the trace of $x_2$ we see that $(i,j,k,M_1)$ is one of the following:
$$(2,14,10,L_X(18)), (16,4,8,L_X(18)), (16,4,10,U), (16,14,8,U), (2,4,8,U).$$
In all of these cases, $x_3$ has trace $2$ on $V$, which is compatible with Proposition \ref{p:trace}. If $V|_{X} = W(16) \oplus W(4) \oplus W(8) \oplus L_X(18)$ then there is an element $y \in X$ of order $5$ with eigenvalues $[I_{25},\omega I_{27}, 
\omega^2 I_{27},\omega^3 I_{27},\omega^4 I_{27}]$ on $V$, but one checks that there are no elements in $G$ that act on $V$ in this way (for example, see \cite[Table 6]{CG}), so this possibility is ruled out. 

If $V|_{X}$ is one of 
$$W(8) \oplus W(4) \oplus W(2) \oplus U,\;\; W(16) \oplus W(10) \oplus W(4) \oplus U,$$
$$W(16) \oplus W(14) \oplus W(8) \oplus U,$$
then $X$ stabilizes the $1$-dimensional subalgebra of $V$ spanned by the vector 
$$w=e_{1122111}-e_{1112211}+e_{0112221}.$$ 
Indeed, $X$ stabilizes ${\rm soc}(U) = L_X(0)$, which is spanned by a vector in $C_V(x) \cap E_0$. But one checks that $C_V(x) \cap E_0 = \la w \ra$ so we are in case (ii) in the statement of the theorem.

Finally, suppose $V|_{X} = W(2) \oplus W(14) \oplus W(10) \oplus L_X(18)$, which is 
compatible with the containment of $X$ in an $A_1$-type subgroup of $G$ (see Table \ref{t:dec}). Let $W$ be the $L_X(2)$ summand in the socle of $V|_{X}$ and let $\{w_2,w_0,w_{-2}\}$ be a standard basis. In the usual way we obtain 
\begin{eqnarray*}
w_2 \!\!\! & = & \!\!\! a_1(e_1+e_2+e_3+e_4+e_5+e_6+e_7) \\
w_0  \!\!\! & = & \!\!\! a_2(h_1+2h_2+12h_3+14h_4+5h_5+6h_6+17h_7) \\
& & \!\!\! + a_3(e_{1122111}-e_{1112211}+e_{0112221}) \\
w_{-2}  \!\!\! & = & \!\!\! a_4(9f_1+18f_2+13f_3+12f_4+7f_5+16f_6+f_7) \\
& & \!\!\! + a_5(e_{1122110}-e_{1112210}+13e_{1112111}+12e_{0112211}) +a_6(e_{2234321})
\end{eqnarray*}
and we may assume $a_1=1$. By considering the action of $x$ on $W$, together with the condition in (\ref{e:wm}), we deduce that $a_2=2$, $a_4=11$, $a_5=16a_3$ and $a_6 = 13a_3^2$, so we have 
\begin{eqnarray*}
w_2  \!\!\! & = & \!\!\! e_1+e_2+e_3+e_4+e_5+e_6+e_7 \\
w_0 \!\!\! & = & \!\!\! 2(h_1+2h_2+12h_3+14h_4+5h_5+6h_6+17h_7) \\
 & & \!\!\! + \gamma(e_{1122111}-e_{1112211}+e_{0112221}) \\
w_{-2}  \!\!\! & = & \!\!\! 11(9f_1+18f_2+13f_3+12f_4+7f_5+16f_6+f_7) \\
& & \!\!\! + 16\gamma(e_{1122110}-e_{1112210}+13e_{1112111}+12e_{0112211}) +13\gamma^2(e_{2234321})
\end{eqnarray*}
for some $\gamma \in K$. Set
$$e = w_2, \; h = -2(2(h_1+2h_2+12h_3+14h_4+5h_5+6h_6+17h_7)),$$
\begin{eqnarray*}
f \!\!\! & = & \!\!\! -11(9f_1+18f_2+13f_3+12f_4+7f_5+16f_6+f_7) \\
& = & \!\!\! 15f_1+11f_2+9f_3+f_4+18f_5+14f_6+8f_7
\end{eqnarray*}
and
\begin{eqnarray*}
y \!\!\! & = & \!\!\! e_{1122111}-e_{1112211}+e_{0112221} \\
z_1 \!\!\! & = & \!\!\! 8(e_{1122110}-e_{1112210}+13e_{1112111}+12e_{0112211}) \\
z_2 \!\!\! & = & \!\!\! 11e_{2234321}
\end{eqnarray*}
Then $(e,h,f)$ is an $\mathfrak{sl}_2$-triple over $\Z$, but we cannot directly apply Proposition \ref{p:new2}. However, a minor modification of the argument in the 
proof of that proposition will work. 

First observe that $y\in ({\mathcal L}_\Z)_{p-1}\cap C_{{\mathcal L}_\Z}(e)$,
$z_1 \in ({\mathcal L}_\Z)_{p-3}$ and $z_2\in ({\mathcal L}_\Z)_{2p-4}$ (in terms of the notation used in the proof of Proposition \ref{p:new2}). Setting $\delta = -2\gamma$, we see that
$$(w_2, -2w_0,-w_{-2}) = ({\bar e}, {\bar h} + \delta {\bar y}, {\bar f}+\delta {\bar z}_1 + \delta^2 {\bar z}_2)$$
is an $\mathfrak{sl}_2$-triple in $\mathcal{L}_K$ for all choices of $\gamma \in K$.
Put $g ={\rm exp}({\rm ad}(\delta y)) \in G$ and note that 
$$g \cdot {\bar e} = \bar{e}, \;\; g \cdot \bar{h} = \bar{h}+\delta [{\bar y},{\bar h}] =\bar {h}+\delta \bar{y},\;\; 
g\cdot \bar{f} = \bar{f}+\delta [{\bar y},{\bar f}] + \frac{1}{2}\delta^2[{\bar y},[{\bar y},{\bar f}]]$$ 
(for the final equality, note that all higher degree terms are zero since the maximum $T$-weight on $\mathcal{L}_{\Z}$ is $2{\rm ht}(\alpha_0) \leqs 2(p-1)$).  
Now calculating (in ${\mathcal L}_\Z$), we have 
$$[h+y,f+z_1+z_2] = -2f+(p-3)z_1+(2p-4)z_2+[y,f]+[y,z_1]$$ 
and passing to ${\mathcal L}_K$, setting $\gamma = -\frac{1}{2}$, we deduce that  
$$-2\bar{f}-3\bar{z}_1-4\bar{z}_2+[{\bar y},{\bar f}]+[{\bar y},{\bar z}_1] = -2({\bar f}+{\bar z}_1+{\bar z}_2).$$
Therefore $[{\bar y},{\bar f}]+[{\bar y},{\bar z}_1] = {\bar z}_1+2{\bar z}_2$ and by comparing $T$-weights we deduce that
$[{\bar y},{\bar f}] = {\bar z}_1$ and $[{\bar y},{\bar z}_1] = 2{\bar z}_2$.
Finally, this implies that 
$$g \cdot \bar{f} = \bar{f}+\delta {\bar z}_1+\frac{1}{2}\delta^2[{\bar y},{\bar z}_1] = \bar{f}+\delta \bar{z}_1+\delta^2 \bar{z}_2$$
and we can now conclude as in the proof of Proposition \ref{p:new2}. In particular, $X$ is contained in an $A_1$-type subgroup of $G$.  

\vs

This completes the proof of Theorem \ref{t:e7}.
 \end{proof}

\section{A reduction for $G=E_8$}\label{s:e8}

In this section we complete the proof of the Reduction Theorem (see Theorem \ref{t:red}). Our main result is the following:

\begin{thm}\label{t:e8}
Let $G$ be a simple algebraic group of type $E_8$ over an algebraically closed field of characteristic $p >0$. Let $X = {\rm PSL}_{2}(p)$ be a subgroup of $G$ containing a regular unipotent element $x$ of $G$ and set $V = {\rm Lie}(G)$. Then one of the following holds:
\begin{itemize}\addtolength{\itemsep}{0.2\baselineskip}
\item[{\rm (i)}] $X$ is contained in an $A_1$-type subgroup of $G$; 
\item[{\rm (ii)}] $p=37$, $V|_{X} = W(34) \oplus W(26) \oplus W(14) \oplus L_X(22) \oplus L_X(2)$ and $X$ stabilizes a non-zero subalgebra of $\la e_{\a} \mid \a \in \Phi^{+}(G)\ra$.
\end{itemize}
\end{thm}

\begin{proof}
First note that $p \geqs 31$. In fact, we may assume $p \geqs 37$ since the case $p=31$ was handled in Section \ref{s:prel} (see Examples \ref{e:e8} and \ref{e:e9}). Recall that 
\begin{equation}\label{e:eige8}
\{\xi^{58},\xi^{46},\xi^{38},\xi^{34},\xi^{26},\xi^{22},\xi^{14},\xi^2\}
\end{equation} 
is the collection of eigenvalues of $s$ on $C_V(x)$ and note that $V|_{X}$ is self-dual.
The Jordan form of $x$ on $V$ is as follows:
\begin{equation}\label{e:6}
\left\{\begin{array}{ll}
\mbox{$[J_{59},J_{47},J_{39},J_{35},J_{27},J_{23},J_{15},J_3]$} & p \geqs 59 \\
\mbox{$[J_{53}^2,J_{39},J_{35},J_{27},J_{23},J_{15},J_3]$} & p = 53 \\
\mbox{$[J_{47}^3,J_{39},J_{27},J_{23},J_{15},J_3]$} & p = 47 \\
\mbox{$[J_{43}^4,J_{35},J_{23},J_{15},J_{3}]$} & p=43 \\
\mbox{$[J_{41}^4,J_{39},J_{27},J_{15},J_{3}]$} & p=41 \\
\mbox{$[J_{37}^6,J_{23},J_{3}]$} & p=37 \\
\mbox{$[J_{31}^8]$} & p=31 
\end{array}\right.
\end{equation}
(see \cite[Table 9]{Lawther}). 

\vs

\noindent \emph{Case 1. $V|_{X}$ is semisimple} 

\vs

By considering the eigenvalues in (\ref{e:eige8}) we deduce that $p \geqs 59$ and 
$$V|_{X} = L_X(58) \oplus L_X(46) \oplus L_X(38) \oplus L_X(34) \oplus L_X(26) \oplus L_X(22) \oplus L_X(14) \oplus L_X(2).$$
Let $W$ be the $L_X(2)$ summand and let $\{w_2,w_0,w_{-2}\}$ be a standard basis. If $p \geqs 61$ then it is straightforward to show that $W$ is an appropriate 
$\mathfrak{sl}_2$-subalgebra and the result follows by applying Proposition \ref{p:new1} (note that if $p=61$ then $\ker((x-1)^3) \cap E_{-2}$ is $2$-dimensional, but this does not cause any special difficulties). Now assume $p=59$. Here $\ker((x-1)^2) \cap E_{0}$ and $\ker((x-1)^3) \cap E_{-2}$ are both $2$-dimensional and we get
\begin{eqnarray*}
w_2 \!\!\! & = & \!\!\! a_1(e_1+e_2+e_3+e_4+e_5+e_6+e_7+e_8) \\
w_0 \!\!\! & = & \!\!\! a_2(h_1+22h_2+52h_3+35h_4+46h_5+48h_6+41h_7+25h_8) \\
& & \!\!\! + a_3(e_{23465432}) \\
w_{-2} \!\!\! & = & \!\!\! a_4(26f_1+41f_2+54f_3+25f_4+16f_5+9f_6+4f_7+f_8) + a_5(e_{23465431}).
\end{eqnarray*}
Set $a_1=1$ and consider the action of $x$ on $W$ (see (\ref{e:xmat})). We  deduce that $a_2 = 13$, $a_4=1$ and $a_5 = 2a_3$, so 
\begin{eqnarray*}
w_2 \!\!\! & = & \!\!\! e_1+e_2+e_3+e_4+e_5+e_6+e_7+e_8 \\
w_0 \!\!\! & = & \!\!\! 13(h_1+22h_2+52h_3+35h_4+46h_5+48h_6+41h_7+25h_8) \\
& & \!\!\! + \gamma(e_{23465432}) \\
w_{-2} \!\!\! & = & \!\!\! 26f_1+41f_2+54f_3+25f_4+16f_5+9f_6+4f_7+f_8 + 2\gamma(e_{23465431})
\end{eqnarray*}
for some $\gamma \in K$. Set 
$$e = w_2, \; h = -2(13(h_1+22h_2+52h_3+35h_4+46h_5+48h_6+41h_7+25h_8))$$
and 
\begin{eqnarray*}
f \!\!\! & = & \!\!\! -(26f_1+41f_2+54f_3+25f_4+16f_5+9f_6+4f_7+f_8) \\
& = & \!\!\! 33f_1+18f_2+5f_3+34f_4+43f_5+50f_6+55f_7+58f_8.
\end{eqnarray*}
Then $(e,h,f)$ is an $\mathfrak{sl}_2$-triple over $\Z$ (see the proof of \cite[Proposition 2.4]{T}) and by applying Proposition \ref{p:new2} (with $y = e_{23465432}$ and $z=e_{23465431}$) we deduce that $X$ is contained in an $A_1$-type subgroup of $G$.

\vs

\noindent \emph{Case 2. ${\rm rad}(V|_{X}) \ne 0$, $p \geqs 53$} 

\vs

If $p \geqs 61$ then the dimension of each indecomposable summand of $V|_{X}$ is at least $60$, which implies that the Jordan form of $x$ has a block of size $n \geqs 60$. This is a contradiction. 

Next assume $p=59$. Suppose $W$ is a reducible indecomposable non-projective summand of $V|_{X}$, so $\dim W \geqs 58$ (see Corollary \ref{c:bound}). In view of (\ref{e:6}) and Lemma \ref{l:jf}, we deduce that $x$ has Jordan form $[J_{59},J_i]$ on $W$ for some odd integer $i$ between $3$ and $47$. But this implies that $\dim W$  is even, so Corollary \ref{c:2step} implies that $W$ has at least four composition factors and thus $i \geqs 57$ (again, by Corollary \ref{c:bound}). This is a contradiction. Therefore, we may assume that each indecomposable summand of $V|_{X}$ is either simple or projective. Clearly,
$$V|_{X} = U \oplus L_X(46) \oplus L_X(38) \oplus L_X(34) \oplus L_X(26) \oplus L_X(22) \oplus L_X(14) \oplus L_X(2)$$
is the only possibility. However, this implies that $x_2$ has trace $-4$ on $V$, which contradicts Proposition \ref{p:trace}.

Now assume $p=53$. As in the previous case, by applying Lemma \ref{l:jf} and Theorem \ref{t:structure}, and by appealing to the self-duality of $V|_{X}$, it is straightforward to reduce to the case where the indecomposable summands of $V|_{X}$ are either simple or projective. Moreover, by considering the eigenvalues in (\ref{e:eige8}), we deduce that 
$$V|_{X} = W(i) \oplus L_X(38) \oplus L_X(34) \oplus L_X(26) \oplus L_X(22) \oplus L_X(14) \oplus L_X(2)$$
with $i \in \{6, 46\}$. By computing the trace of $x_3 \in X$, it follows that $i=46$. It is now entirely straightforward to show that the $L_X(2)$ summand of $V|_{X}$ is an appropriate $\mathfrak{sl}_2$-subalgebra and the result follows via Proposition \ref{p:new1}.

\vs

\noindent \emph{Case 3. ${\rm rad}(V|_{X}) \ne 0$, $p = 47$} 

\vs

As in the previous case, we can quickly reduce to the situation where each indecomposable summand of $V|_{X}$ is simple or projective, in which case 
$$V|_{X} = W(i) \oplus M_2 \oplus L_X(38) \oplus L_X(26) \oplus L_X(22) \oplus L_X(14) \oplus L_X(2)$$
with $i \in \{12,34\}$ and $M_1 \in \{L_X(46),U\}$. By computing the trace of $x_2$ we deduce that $i=34$ and $M_1 = L_X(46)$, in which case $V|_{X}$ is compatible with the containment of $X$ in an $A_1$-type subgroup of $G$ (see Table \ref{t:dec}). 

As usual, let $W$ be the $L_X(2)$ summand of $V|_{X}$ and let $\{w_2,w_0,w_{-2}\}$ be a standard basis for $W$. We get
\begin{eqnarray*}
w_2  \!\!\! & = & \!\!\! a_1(e_1+e_2+e_3+e_4+e_5+e_6+e_7+e_8) \\
w_0 \!\!\! &  = & \!\!\! a_2(h_1+26h_2+3h_3+6h_4+31h_5+10h_6+37h_7+18h_8) \\
& & \!\!\! + a_3(e_{23354321}- e_{22454321}) \\
w_{-2} \!\!\! & = & \!\!\! a_4(34f_1+38f_2+8f_3+16f_4+20f_5+11f_6+36f_7+f_8) \\
& & \!\!\! + a_5(e_{22354321} - 2e_{13354321}).
\end{eqnarray*}
We may set $a_1=1$. By considering the action of $x$ on this basis we deduce that $a_2 = 1$, $a_4=36$ and $a_5=45a_3$, so 
\begin{eqnarray*}
w_2  \!\!\! & = & \!\!\! e_1+e_2+e_3+e_4+e_5+e_6+e_7+e_8 \\
w_0 \!\!\! & = & \!\!\! h_1+26h_2+3h_3+6h_4+31h_5+10h_6+37h_7+18h_8 \\
& & \!\!\! + \gamma(e_{23354321}- e_{22454321}) \\
w_{-2} \!\!\! & = & \!\!\! 36(34f_1+38f_2+8f_3+16f_4+20f_5+11f_6+36f_7+f_8) \\
& & \!\!\! +  45\gamma(e_{22354321} - 2e_{13354321})
\end{eqnarray*}
for some $\gamma \in K$. One now checks that the relations in (\ref{e:br}) are satisfied and thus $W$ is an $\mathfrak{sl}_2$-subalgebra of $V$. Set
$$e = w_2, \; h = -2(h_1+26h_2+3h_3+6h_4+31h_5+10h_6+37h_7+18h_8)$$
and
\begin{eqnarray*}
f \!\!\! & = & \!\!\! - 36(34f_1+38f_2+8f_3+16f_4+20f_5+11f_6+36f_7+f_8) \\
& = & \!\!\! 45f_1+42f_2+41f_3+35f_4+32f_5+27f_6+20f_7+11f_8
\end{eqnarray*}
Then $(e,h,f)$ is an $\mathfrak{sl}_2$-triple over $\Z$ (see the proof of \cite[Proposition 2.4]{T}) and we can use  Proposition \ref{p:new2} to conclude that $X$ is contained in an $A_1$-type subgroup of $G$.

\vs

\noindent \emph{Case 4. ${\rm rad}(V|_{X}) \ne 0$, $p = 43$} 

\vs

By arguing in the usual manner, it is straightforward to reduce to the case where each indecomposable summand of $V|_{X}$ is either simple or projective. By considering the eigenvalues in (\ref{e:eige8}), we deduce that 
$$V|_{X} = W(i) \oplus W(j) \oplus L_X(34) \oplus L_X(22) \oplus L_X(14) \oplus L_X(2)$$
with $i \in \{4,38\}$ and $j \in \{16,26\}$. By computing the trace of $x_2$, we see that $(i,j) = (38,26)$ is the only option, in which case $V|_{X}$ is compatible with the desired containment of $X$ in an $A_1$-type subgroup of $G$. As usual, we now construct the summand $W=L_X(2)$ of $V|_{X}$ in terms of a standard basis $\{w_2,w_0,w_{-2}\}$; it is easy to show that $W$ is an appropriate $\mathfrak{sl}_2$-subalgebra and we can conclude by applying Proposition \ref{p:new1}.

\vs

\noindent \emph{Case 5. ${\rm rad}(V|_{X}) \ne 0$, $p = 41$} 

\vs

First assume that $V|_{X}$ has a reducible indecomposable non-projective summand $W$. In the usual way, we deduce that the Jordan form of $x$ on $W$ is one of the following:
$$\left\{\begin{array}{l}
\mbox{$[J_{41}^4,J_{27}]$}, [J_{41}^4,J_{15}] \\
\mbox{$[J_{41}^3,J_{3}]$} \\
\mbox{$[J_{41}^2,J_{27}]$}, [J_{41}^2,J_{15}], [J_{41}^2,J_{3}] \\
\mbox{$[J_{41},J_{39}]$},
\end{array}\right.$$
If the Jordan form is either $[J_{41}^4,J_{27}]$ or $[J_{41}^4,J_{15}]$ then there is a unique such summand. Moreover, $W$ has an odd number of composition factors and it is easy to see that this is incompatible with the self-duality of $V|_{X}$. Similar reasoning rules out the cases where $x$ has Jordan form $[J_{41}^2,J_i]$. Finally, suppose $x$ has Jordan form $[J_{41}^3, J_3]$ or $[J_{41},J_{39}]$. Here the self-duality of $V|_{X}$ implies that 
$${\rm soc}(W) \cong W/{\rm soc}(W) =L_X(22) \oplus L_X(20) \oplus L_X(18)$$
or
$${\rm soc}(W) \cong W/{\rm soc}(W) =L_X(20) \oplus L_X(18),$$
respectively. However, the existence of such a summand would mean that $\xi^{20}$ is an eigenvalue of $s$ on $C_V(x)$, which is not the case (see (\ref{e:eige8})). Therefore, we conclude that every indecomposable summand of $V|_{X}$ is either simple or projective. More precisely, in view of (\ref{e:eige8}), it follows that
$$V|_{X} = W(i) \oplus W(j) \oplus L_X(38) \oplus L_X(26) \oplus L_X(14) \oplus L_X(2)$$
with $i \in \{6,34\}$ and $j \in \{18,22\}$. 
By computing the trace of $x_3$ we deduce that $(i,j) = (34,22)$, in which case $V|_{X}$ is compatible with the containment of $X$ in an $A_1$-type subgroup. 

Let $W$ be the $L_X(2)$ summand of $V|_{X}$ and let $\{w_2,w_0,w_{-2}\}$ be a standard basis. In the usual manner we deduce that
\begin{eqnarray*}
w_2  \!\!\! & = & \!\!\! a_1(e_1+e_2+e_3+e_4+e_5+e_6+e_7+e_8) \\
w_0 \!\!\! & = & \!\!\! a_2(h_1+30h_2+10h_3+27h_4+22h_5+25h_6+36h_7+14h_8) \\
w_{-2} \!\!\! & = & \!\!\! a_3(3f_1+8f_2+30f_3+40f_4+25f_5+34f_6+26f_7+f_8) \\
& & \!\!\! + a_4(e_{22343221} - e_{12343321} + e_{12244321})
\end{eqnarray*}
Set $a_1=1$. By considering the action of $x$ on this basis we get $a_2 = 36$ and $a_3 = 24$. Finally, one can check that the condition in (\ref{e:wm}) implies that $a_4=0$ and now the desired result follows from Proposition \ref{p:new1}.

\vs

To complete the proof of the theorem, we may assume that $p=37$ (recall that the case $p=31$ was handled earlier in Examples \ref{e:e8} and \ref{e:e9}).

\vs

\noindent \emph{Case 6. ${\rm rad}(V|_{X}) \ne 0$, $p = 37$} 

\vs

As usual, let us first assume that $V|_{X}$ has a reducible indecomposable non-projective summand $W$. By applying Lemma \ref{l:jf} and Theorem \ref{t:structure}, we deduce that the Jordan form of $x$ on $W$ is one of the following:
$$[J_{37}^6,J_{23}], \; [J_{37}^4,J_{23}], \; [J_{37}^3,J_{3}], \; [J_{37}^2,J_{23}], \;  [J_{37}^2,J_{3}].$$
In fact, the self-duality of $V|_{X}$ implies that $[J_{37}^3,J_{3}]$ is the only possibility, with
$${\rm soc}(W) \cong W/{\rm soc}(W) = L_X(16) \oplus L_X(18) \oplus L_X(20).$$
But if this is a summand of $V|_{X}$ then $\xi^{20}$ is an eigenvalue of $s$ on $C_V(x)$, contradicting (\ref{e:eige8}). Therefore, we have reduced to the case where each indecomposable summand of $V|_{X}$ is simple or projective. Again, by considering (\ref{e:eige8}) we deduce that 
\begin{equation}\label{e:decf}
V|_{X} = W(i) \oplus W(j) \oplus W(k) \oplus L_X(22) \oplus L_X(2)
\end{equation}
with $i \in \{2,34\}$, $j \in \{10,26\}$ and $k \in \{14,22\}$. 

We claim that $(i,j,k) = (34,26,14)$, in which case the decomposition of $V|_{X}$ is compatible with the containment of $X$ in an $A_1$ subgroup of $G$. One can check that all of the eight  decompositions above are compatible with the trace of $x_2$ and $x_3$, so we consider the traces of elements of larger order. Let $y \in X$ be an element of order $19$. In each case it is straightforward to compute the eigenvalues of $y$ on $V$. Using Litterick's algorithm in \cite{AJL}, we can compute the eigenvalues on $V$ of every element in $G$ of order $19$ and in this way we deduce that $(i,j,k) = (34,26,14)$ as claimed.

Let $W$ be the $L_X(2)$ summand of $V|_{X}$ with standard basis $\{w_2,w_0,w_{-2}\}$. The spaces 
$$\ker(x-1) \cap E_2,\;\; \ker((x-1)^2) \cap E_0,\;\; \ker((x-1)^3) \cap E_{-2}$$
have respective dimensions $2$, $2$ and $3$, which gives
\begin{eqnarray*}
w_2  \!\!\! & = & \!\!\! a_1(e_1+e_2+e_3+e_4+e_5+e_6+e_7+e_8)  \\
& & \!\!\! + a_2(e_{22343221} - e_{12343321} + e_{12244321}) \\
w_0 \!\!\! & = & \!\!\! a_3(h_1+24h_2+6h_3+15h_4+4h_5+34h_6+31h_7+32h_8)  \\
& & \!\!\! +a_4(e_{22343211}+24e_{12343221}+25e_{12243321}) \\
w_{-2} \!\!\! & = & \!\!\! a_5(22f_1+10f_2+21f_3+34f_4+14f_5+8f_6+16f_7+f_8) \\
& & \!\!\! + a_6(e_{22343210} +13e_{12243221} - e_{12233321}) \\
& & \!\!\! + a_7(e_{12343211} + 23e_{12243221} + 2e_{12233321})
\end{eqnarray*}
By considering the action of $x$ on this basis, we deduce that $a_3 = 28a_1$, $a_4 =  3a_2$, $a_5=16a_1$ and $a_7 = 6a_2+a_6$. The condition in (\ref{e:wm}) yields the equations
\begin{eqnarray*}
16a_1^2a_2 + 3a_1^2a_6 \!\!\! & = & \!\!\! 0 \\
19a_1^2a_2 + 6a_1^2a_6 \!\!\! & = & \!\!\! 0
\end{eqnarray*}
If $a_1 \ne 0$ then these equations imply that $a_2 = a_6 = 0$, so we can set $a_1=1$ and then apply Proposition \ref{p:new1} to show that $X$ is contained in an $A_1$-type subgroup. On the other hand, if $a_1=0$ then we may assume $a_2=1$, so 
\begin{eqnarray*}
w_2  \!\!\! & = & \!\!\! e_{22343221} - e_{12343321} + e_{12244321} \\
w_0 \!\!\! & = & \!\!\! 3(e_{22343211}+24e_{12343221}+25e_{12243321}) \\
w_{-2} \!\!\! & = & \!\!\! a_6(e_{22343210} +13e_{12243221} - e_{12233321}) \\
& & \!\!\! + (6+a_6)(e_{12343211} + 23e_{12243221} + 2e_{12233321})
\end{eqnarray*}
It is straightforward to check that $W \subseteq \la e_{\a} \mid \a \in \Phi^+(G)\ra$ is a subalgebra and this puts us in case (ii) of the theorem.

\vs

This completes the proof of Theorem \ref{t:e8}.
 \end{proof}

\section{Proof of Theorem \ref{t:main}}\label{s:final}

In this final section we complete the proof of Theorem \ref{t:main}. In view of Theorem \ref{t:g2}, we may assume that $G$ is of type $F_4$, $E_6$, $E_7$ or $E_8$. Moreover, by our work in Sections \ref{s:f4}--\ref{s:e8}, it remains to handle the cases appearing in Table \ref{tab:red}. In each of these cases, $X$ stabilizes a non-zero subalgebra $W \subseteq \la e_{\a} \mid \a \in \Phi^{+}(G)\ra$ of $V = {\rm Lie}(G)$ and by applying Proposition \ref{l:bt} we can assume that $X$ is contained in a proper parabolic subgroup $P=QL$ of $G$ with unipotent radical $Q$ and Levi factor $L$. The following result, when combined with Theorem \ref{t:main0}, completes the proof of Theorem \ref{t:main}. (Recall that Craven \cite{Craven} has constructed a subgroup 
$X$ satisfying the conditions in parts (ii) and (iii) of Theorem \ref{t:main}, and he has established its uniqueness up to conjugacy; see Remark \ref{r:main}(b).)

\begin{thm}\label{t:par}
Let $G$ be a simple exceptional algebraic group of adjoint type over an algebraically closed field of characteristic $p>0$. Let $X = {\rm PSL}_{2}(p)$ be a subgroup of $G$ containing a regular unipotent element of $G$ and let $V = {\rm Lie}(G)$ be the adjoint module. If $X$ is contained in a proper parabolic subgroup $P=QL$ of $G$, then either
\begin{itemize}\addtolength{\itemsep}{0.2\baselineskip}
\item[{\rm (i)}] $G=E_6$, $p=13$, $L'=D_5$ and $V|_{X} = W(10)^2 \oplus W(4)$; or
\item[{\rm (ii)}] $G=E_7$, $p=19$, $L'=E_6$ and $V|_{X} = W(16) \oplus W(14) \oplus W(8) \oplus U$. 
\end{itemize}
\end{thm}

\begin{proof}
We may assume that $P$ is minimal with respect to containing $X$. Let $\pi:P\to P/Q$ be the quotient map and identify $L$ with $P/Q$. By arguing as in the first paragraph in the proof of Theorem \ref{t:main0} (see the end of Section \ref{s:prel}), we deduce that $\pi(X)$ is contained in an $A_1$-type subgroup $H$ of $L'$. In addition, Theorem \ref{t:main0} implies that $X$ is not contained in an $A_1$-type subgroup of $G$, so  $(G,p,V|_{X})$ must be one of the cases in Table \ref{tab:red}. As noted in the proof of Theorem \ref{t:main0}, the composition factors of $V|_{H}$ can be read off from \cite[Tables 1--5]{LT99} and this imposes restrictions on the composition factors of $V|_{X}$. By considering each possibility for $(G,L')$ in turn,  comparing composition factors with Table \ref{tab:red}, we will show that the cases labelled (i) and (ii) in the statement of the theorem are the only compatible options.

First assume $(G,p) = (F_4,13)$, so the composition factors of $V|_{X}$ are $L_X(10)^3$, $L_X(8)$, $L_X(2)^3$ and $L_X(0)$. By inspecting \cite[Table 2]{LT99} it is easy to see that  there is no compatible Levi subgroup $L$. Similarly, if $(G,p) = (E_8,37)$ then the composition factors of $V|_{X}$ are given in (\ref{e:e8dee}) and thus we can eliminate this case by repeating the argument in the proof of Theorem \ref{t:main0}.

Next suppose $(G,p) = (E_6,13)$. The three possibilities for $V|_{X}$ (and their composition factors) are as follows:
\begin{eqnarray*}
W(10) \oplus W(8) \oplus W(2): & & \hspace{-5mm} L_X(10)^3, L_X(8)^3, L_X(4), L_X(2)^4, L_X(0) \\
W(10) \oplus W(4) \oplus W(2): & & \hspace{-5mm} L_X(10)^3, L_X(8)^2, L_X(6), L_X(4)^2, L_X(2)^3, L_X(0) \\
W(10) \oplus W(10) \oplus W(4): & & \hspace{-5mm} L_X(10)^4, L_X(8), L_X(6), L_X(4)^2, L_X(2)^2, L_X(0)^2 
\end{eqnarray*}
In all three cases, we see that $V|_{X}$ has at most two trivial composition factors, so \cite[Table 3]{LT99} implies that 
$$\mbox{$L' = A_2A_1^2, A_2^2A_1, A_4A_1$ or $D_5$.}$$
If $L'=A_2A_1^2$ then $V|_{X}$ has five or more $L_X(2)$ factors, which is incompatible with all three possibilities for $V|_{X}$. Similarly, if $L'=A_4A_1$ then there are too many $L_X(4)$ factors, and we can rule out $L'=A_2^2A_1$ because we would get $L_X(1)$ composition factors, which is absurd. Finally, suppose $L' = D_5$. Since the Weyl module $W_X(14)$ has an $L_X(10)$ composition factor, we see that $V|_{X}$ has four such factors and thus 
$$V|_{X} = W(10) \oplus W(10) \oplus W(4)$$ 
is the only option. 

Finally, let us assume $(G,p) = (E_7,19)$. The three possibilities for $V|_{X}$ are as follows:
\begin{eqnarray*}
W(8) \oplus W(4) \oplus W(2) \oplus U: && \hspace{-5mm} L_X(16)^2, L_X(14)^2, L_X(12), L_X(10), L_X(8)^3, 
L_X(4)^2, \\
& & \hspace{-5mm}  L_X(2)^2, L_X(0)^2 \\
W(16) \oplus W(10) \oplus W(4) \oplus U: & & \hspace{-5mm} L_X(16)^3, L_X(14), L_X(12), L_X(10)^2, L_X(8), L_X(6), \\
&& \hspace{-5mm} L_X(4)^2, L_X(2), L_X(0)^3 \\
W(16) \oplus W(14) \oplus W(8) \oplus U: && \hspace{-5mm} L_X(16)^3, L_X(14)^2, L_X(10), L_X(8)^3, L_X(4), L_X(2)^2, \\
& & \hspace{-5mm} L_X(0)^3 
\end{eqnarray*}
By inspecting \cite[Table 4]{LT99}, counting the number of trivial composition factors, we quickly reduce to a small number of possibilities for $L'$. By considering non-trivial composition factors, it is straightforward to reduce further to the case $L'=E_6$. For example, we can rule out $L' = A_6$ because there would be too many $L_X(4)$ factors. Similarly, $L' = D_5A_1$ is out because we would have an $L_X(6)$ and at least three $L_X(8)$ factors, which is not  compatible with any of the three possibilities above. We can rule out $L' = D_6$ because it would imply that $V|_{X}$ has an $L_X(5)$ factor. Finally, suppose $L' = E_6$. Here  $V|_{X}$ has at least three $L_X(16)$ and $L_X(8)$ composition factors, so 
$$V|_{X} = W(16)\oplus W(14) \oplus W(8) \oplus U$$ 
is the only possibility. 
 \end{proof}

\section*{Acknowledgements} 

Testerman was supported by the Fonds National Suisse de la Recherche Scientifique grants $200021\_146223$ and $200021\_153543$.  Burness thanks the Section de Math\'{e}matiques and the Centre Interfacultaire Bernoulli at EPFL for their generous hospitality. It is a pleasure to thank David Craven and Jacques Th\'{e}venaz for many useful discussions. We also thank Bob Guralnick, Gunter Malle and Iulian Simion for helpful comments on an earlier draft of this paper. Finally, we thank an anonymous referee for carefully reading the paper and for making several suggestions that have improved both the accuracy of the paper and the clarity of the exposition.

\end{document}